\documentclass[11pt,fleqn]{amsart}
\usepackage[utf8]{inputenc}

\author{Tomasz \.Zuchowski}
\title{The Nikodym property and filters on $\omega$}
\address{Mathematical Institute, University of Wrocław, Wrocław, Poland.}
\email{tomasz.zuchowski@math.uni.wroc.pl}
\thanks{The author was supported by the Austrian Science Fund (FWF), research grant no. I 5918-N}



\usepackage{mathtools}
\usepackage{amssymb}
\usepackage{amsthm}
\usepackage{amsmath}
\usepackage[colorlinks=true, citecolor=, linkcolor=]{hyperref}
\usepackage{soul}
\usepackage{enumerate}
\usepackage{thmtools}
\usepackage{thm-restate}
\usepackage{mathrsfs}
\usepackage[T1]{fontenc}

\usepackage[left=3cm, right=3cm, bottom=2.5cm, top=3cm]{geometry}

\newtheorem{theorem}{Theorem}[section]
\newtheorem{lemma}[theorem]{Lemma}
\newtheorem{proposition}[theorem]{Proposition}

\newtheorem{question}[theorem]{Question}
\newtheorem{corollary}[theorem]{Corollary}

\theoremstyle{definition}
\newtheorem{definition}[theorem]{Definition}

\theoremstyle{remark}
\newtheorem{remark}[theorem]{Remark}

\numberwithin{equation}{section}

\renewcommand{\setminus}{\mathbin{\backslash}}%

\newcommand{\con}{\mathfrak c}
\newcommand{\eps}{\varepsilon}

\newcommand{\BB}{\protect{\mathcal B}}
\newcommand{\AAA}{\mathcal A}

\newcommand{\FF}{{\mathcal F}}
\newcommand{\GG}{{\mathcal G}}
\newcommand{\II}{{\mathcal I}}
\newcommand{\JJ}{{\mathcal J}}

\newcommand{\NN}{{\mathcal N}}

\newcommand{\PP}{{\mathcal P}}
\newcommand{\SSS}{{\mathcal S}}

\newcommand{\HH}{{\mathcal H}}
\newcommand{\ZZ}{{\mathcal Z}}

\newcommand{\er}{\mathbb R}

\newcommand{\sm}{\setminus}

\newcommand{\sub}{\subseteq}

\newcommand{\unb}{\mathfrak b}
\newcommand{\dom}{\mathfrak d}

\newcommand{\cantor}{2^{\omega}}

\newcommand{\Baire}{\omega^{\omega}}

\newcommand{\hrusak}{Hru\v{s}\'ak}
\newcommand{\hernandez}{Hern\'andez-Hern\'andez}
\newcommand{\guzman}{Guzm\'an-Gonz\'alez}
\newcommand{\meza}{Meza-Alc\'antara}
\newcommand{\katetov}{Kat\v{e}tov}
\newcommand{\cech}{\v{C}ech}
\newcommand{\erdos}{Erd\H{o}s}
\newcommand{\stevo}{Todor\v{c}evi\'c}

\newcommand{\seqnN}[1]{\langle#1\colon n\in\omega\rangle}
\newcommand{\seqnNbig}[1]{\big\langle#1\colon n\in\omega\big\rangle}
\newcommand{\seqkNbig}[1]{\big\langle#1\colon k\in\omega\big\rangle}
\newcommand{\seqnNP}[1]{\langle#1\colon n\in\omega_+\rangle}

\newcommand{\seqkN}[1]{\langle#1:\ k\in\omega\rangle}

\DeclareMathOperator{\at}{at}

\DeclareMathOperator{\interior}{int}

\DeclareMathOperator{\cov}{cov}
\DeclareMathOperator{\tr}{tr}

\DeclareMathOperator{\finI}{Fin}
\DeclareMathOperator{\clopen}{Clopen}
\DeclareMathOperator{\exh}{Exh}
\DeclareMathOperator{\supp}{supp}

\begin{document}

\begin{abstract}
For a free filter $F$ on $\omega$, let $N_F=\omega\cup\{p_F\}$, where $p_F\not\in\omega$, be equipped with the following topology: every element of
$\omega$ is isolated whereas all open neighborhoods of $p_F$ are of the form $A\cup\{p_F\}$ for $A\in F$. The aim of this paper is to study spaces of the form $N_F$ in the context of the Nikodym property of Boolean algebras. By $\mathcal{AN}$ we denote the class of all those ideals $\mathcal{I}$ on $\omega$ such that for the dual filter $\mathcal{I}^*$ the space $N_{\mathcal{I}^*}$ carries a sequence $\langle\mu_n\colon n\in\omega\rangle$ of finitely supported signed measures such that $\|\mu_n\|\rightarrow\infty$ and $\mu_n(A)\rightarrow 0$ for every clopen subset $A\subseteq N_{\mathcal{I}^*}$. We prove that $\mathcal{I}\in\mathcal{AN}$ if and only if there exists a density submeasure $\varphi$ on $\omega$ such that
$\varphi(\omega)=\infty$ and $\mathcal{I}$ is contained in the exhaustive ideal $\mbox{Exh}(\varphi)$. 
Consequently, we get that if $\mathcal{I}\subseteq\mbox{Exh}(\varphi)$ for some density submeasure $\varphi$ on $\omega$ such that $\varphi(\omega)=\infty$ and $N_{\mathcal{I}^*}$ is homeomorphic to a subspace of the Stone space $St(\mathcal{A})$ of a given Boolean algebra $\mathcal{A}$, then $\mathcal{A}$ does not have the Nikodym property. 

We observe that 
each $\mathcal{I}\in\mathcal{AN}$ is Kat\v{e}tov below the asymptotic density zero ideal $\mathcal{Z}$, and prove that the class $\mathcal{AN}$ has a subset of size $\mathfrak{d}$ which is dominating with respect to the Kat\v{e}tov order $\leq_K$, but $\mathcal{AN}$ has no $\leq_K$-maximal element. We show that for a density ideal $\mathcal{I}$ it holds $\mathcal{I}\not\in\mathcal{AN}$ if and only if $\mathcal{I}$ is totally bounded if and only if the Boolean algebra $\mathcal{P}(\omega)/\mathcal{I}$ contains a countable splitting family.

Our results shed some new light on differences between the Nikodym property and the Grothendieck property of Boolean algebras.
\end{abstract}

\maketitle

\section{Introduction}

Let $\AAA$ be a Boolean algebra. We say that a sequence $\seqnN{\mu_n}$ of finitely additive signed measures on a Boolean algebra $\AAA$ is \textit{pointwise bounded} if $\sup_{n\in\omega}|\mu_n(A)|<\infty$ for every $A\in\AAA$, and it is \textit{uniformly bounded} if $\sup_{n\in\omega}\|\mu_n\|<\infty$ (see Section $2$ for more definitions).
The Nikodym Uniform Boundedness Theorem implies that if $\AAA$ is $\sigma$-complete, then every pointwise bounded sequence of measures on $\AAA$ is uniformly bounded. This theorem, due to its numerous applications, is one of the most important results in the theory of vector measures, see Diestel and Uhl \cite[Section I.3]{DieUhl}, and hence motivates the following definition. 

\begin{definition}
A Boolean algebra $\AAA$ has the \textit{Nikodym property} if every pointwise bounded sequence of finitely additive signed measures on $\AAA$ is uniformly bounded.
\end{definition}
\noindent Equivalently, $\AAA$ has the Nikodym property if, for every sequence $\seqnN{\mu_n}$ of finitely additive signed measures on $\AAA$ such that $\mu_n(A)\to 0$ for all $A\in\AAA$, we have $\sup_{n\in\omega}\|\mu_n\|<\infty$ (see \cite[Proposition 2.4]{SZ}).

The Nikodym theorem implies that $\sigma$-complete algebras have the Nikodym property. This result has been generalized by many authors who introduced various weaker properties of Boolean algebras that still imply the Nikodym property, see e.g. Haydon \cite{Hay}, Schachermayer \cite{Sch}, Freniche \cite{Fre}. 

On the other hand, countable infinite Boolean algebras do not have the Nikodym property. Indeed, the Stone space $St(\AAA)$ of a countable infinite Boolean algebra $\AAA$ is homeomorphic to an infinite closed subset of the Cantor space $2^{\omega}$, hence it contains a non-trivial sequence $\seqnN{x_n}$ convergent to some $x\in\cantor$. Then, the sequence $\seqnNbig{n(\delta_{x_n}-\delta_x)}$ of Borel measures on $St(\AAA)$ induces a pointwise bounded but not uniformly bounded sequence of finitely additive measures on $\AAA$. This example is a motivation for considering the following type of spaces.

\begin{definition}
Let $F$ be a free filter on $\omega$. Endow the set $N_F =\omega\cup\{p_F\}$, where $p_F$ is a fixed point not belonging to $\omega$, with the topology $\tau_F$ defined in the following way:
\begin{itemize}
\item every point of $\omega$ is isolated in $N_F$ , i.e. $\{n\}$ is open for every $n\in\omega$,
\item a set $U\sub N_F$ is an open neighborhood of $p_F$ if and only if $U=A\cup\{p_F\}$ for some $A\in F$.
\end{itemize}
\end{definition}

For every free filter $F$ the space $N_F$ is a countable non-discrete zero-dimensional space (and so a Tychonoff space). The space $N_{Fr}$, where by $Fr$ we denote the Fr\'echet filter on $\omega$ (see Section $2$), is homeomorphic to a convergent sequence of points in a given Tychonoff space together with its limit point. Thus, for a Boolean algebra $\AAA$, the Stone space $St(\AAA)$ contains a non-trivial convergent sequence if and only if it contains a homeomorphic copy of $N_{Fr}$. Consequently, if $N_{Fr}$ embeds into $St(\AAA)$, then $\AAA$ does not have the Nikodym property. In this paper we study for which other free filters we also have such an implication.

We need to introduce some notions. We say that a Borel measure $\mu$ on a Tychonoff space $X$ is \textit{finitely supported} if it can be written in the form $\mu=\sum_{i=1}^n\alpha_i\delta_{x_i}$ for some distinct points $x_1,\ldots,x_n\in X$ and some real numbers $\alpha_1,\ldots,\alpha_n\in\er$, $n\in\omega$; in this case we write $\|\mu\|=\sum_{i=1}^n|\alpha_i|$ (see again Section $2$). 

\begin{definition}
Let $X$ be a zero-dimensional (Tychonoff) space. We call a sequence $\seqnN{\mu_n}$ of finitely supported measures on $X$ \textit{anti-Nikodym}, or shortly an \textit{AN-sequence}, if $\|\mu_n\|\rightarrow\infty$ and $\mu_n(A)\rightarrow 0$ for every $A\in\clopen(X)$. We say that $X$ has the \textit{finitely supported Nikodym property}
if there is no anti-Nikodym sequence of measures on $X$.

We also say that a free filter $F$ on $\omega$ has the \textit{Nikodym property} if the space $N_F$ has a finitely supported Nikodym property. We denote by $\AAA\NN$ the class of all ideals on $\omega$ whose dual filters do not have the Nikodym property.
\end{definition}

We prove that if $\AAA$ is a Boolean algebra, $F$ is a free filter on $\omega$ which does not have the Nikodym property, and $N_F$ homeomorphically embeds into $St(\AAA)$, then $\AAA$ does not have the Nikodym property (Corollary \ref{corollarynikodymalgebras}). This result generalizes the aforementioned fact about embedding $N_{Fr}$ into the Stone spaces of Boolean algebras.

For any non-trivial convergent sequence $\seqnN{x_n}$ in a given space $N_F$, the sequence $\seqnN{\mu_n}$ of finitely supported non-negative measures defined for each $n\in\omega$ by the equality $\mu_n=n\cdot\delta_{x_n}$ satisfies the conditions $\lim_{n\to\infty}\mu_n(\omega)=\infty$ and $\lim_{n\to\infty}\mu_n(\omega\sm A)=0$ for every $A\in F$. 
We show that the Nikodym property for filters is in fact characterized by similar conditions.

\begin{restatable*}{theorem}{theorempositivemeasures}\label{thm_positive_measures}
A filter $F$ on $\omega$ has the Nikodym property if and only if there is no (disjointly supported) sequence $\seqnN{\mu_n}$ of finitely supported non-negative measures on $N_F$ such that:
\begin{enumerate}[(1)]
\item $\supp(\mu_n)\sub\omega$ for every $n\in\omega$,
\item $\lim_{n\to\infty}\mu_n(\omega)=\infty$,
\item $\lim_{n\to\infty}\,\mu_n(\omega\sm A)=0$ for every $A\in F$.
\end{enumerate}
\end{restatable*}

Theorem \ref{thm_positive_measures} can be translated in the following way, showing a connection between filters without the Nikodym property and exhaustive ideals associated to lower semicontinuous (\textit{lsc}) submeasures on $\omega$.

\begin{restatable*}{theorem}{theoremANexh}\label{thm_AN_exh}
Let $F$ be a filter on $\omega$. Then, the following are equivalent:
\begin{enumerate}[(1)]
\item $F$ does not have the Nikodym property;
\item there is a density submeasure $\varphi$ on $\omega$ such that $\varphi(\omega)=\infty$ and $F\sub\exh(\varphi)^*$;
\item there is a non-pathological lsc submeasure $\varphi$ on $\omega$ such that $\varphi(\omega)=\infty$ and $F\sub\exh(\varphi)^*$.
\end{enumerate}
\end{restatable*}

The \katetov\ order $\leq_K$ is an important tool for comparing the structural complexity of ideals on $\omega$, see e.g. \cite{Hru}. As a corollary to Theorem \ref{thm_AN_exh}, we get that the class of the ideals dual to filters without the Nikodym property is bounded in this order by the ideal $\ZZ$ of subsets of $\omega$ having \textit{asymptotic density zero} (see Section $4$).

\begin{restatable*}{corollary}{corollaryzet}\label{zz}
If $F$ is a filter on $\omega$ without the Nikodym property, then $F^*\leq_K\ZZ$.
\end{restatable*}

Extending \cite[Theorem 3.16]{Try}, we present several conditions for a density ideal equivalent to the statement that its dual filter has the Nikodym property (see Sections \ref{section_ideals} and \ref{section_katetov} for relevant definitions of classes and properties of ideals).

\begin{restatable*}{theorem}{theoremnikodymequiv}\label{nikodym_equiv}
Let $\II$ be the density ideal on $\omega$ generated by a sequence of measures $\seqnN{\mu_n}$. The following are equivalent:
\begin{enumerate}
\item $\II$ is isomorphic to an \erdos--Ulam ideal.
\item $\II$ does not satisfy the Bolzano--Weierstrass property.
\item The measures $\mu_n$ fulfill the following conditions:
\begin{enumerate}[a)]
\item $\sup_{n\in\omega}\|\mu_n\|< \infty$,
\item $\lim_{n\to\infty} \at^+(\mu_n) = 0$.
\end{enumerate} 
\item The Boolean algebra $\PP(\omega)/\II$ contains a countable splitting family.
\item $\II$ is totally bounded.
\item $\II\equiv_K\ZZ$.
\item $\II^*$ has the Nikodym property.
\end{enumerate}
\end{restatable*}

We make an analysis of the cofinal structure of the family $\AAA\NN$ ordered by $\leq_K$ and we prove that there is a family of size equal to the dominating number $\dom$ which is $\leq_K$-dominating in $\AAA\NN$, but there are no $\leq_K$-maximal elements in $\AAA\NN$.

\begin{restatable*}{theorem}{corollarynikodymdominating}\label{dom}
There exists a family $\FF\sub\AAA\NN$ of size $\dom$ consisting of density ideals and having the property that for any $\II\in\AAA\NN$ there is $\JJ\in\FF$ such that $\II\leq_K\JJ$.
\end{restatable*}

\begin{restatable*}{theorem}{theoremnomaximal}\label{thm_no_maximal}
For every $\II\in\AAA\NN$ there is $\JJ\in\AAA\NN$ such that $\II<_K\JJ$.
\end{restatable*}

Using Theorem \ref{thm_AN_exh} and results from \cite{GuzMez} and \cite{Kwela}, we show that the classes of those filters $F$ on $\omega$ for which the spaces $N_F$ (do not) have the Nikodym property are large. First, there are families $\HH_1,\HH_2\sub\AAA\NN$ of size continuum $\con$, consisting respectively of non-isomorphic summable ideals and of non-isomorphic density ideals (Corollaries \ref{many_summable} and \ref{many_density}). Next, there are families $\FF_1$ and $\FF_2$ of size $2^{\con}$, consisting of non-isomorphic filters on $\omega$ whose dual ideals are tall and which respectively have and do not have the Nikodym property (Corollary \ref{many_many}). Those results yield the following corollary.

\begin{restatable*}{corollary}{corollarymanyAN}
There is a family $\FF$ of $2^{\con}$ many pairwise non-homeomorphic countable non-discrete spaces without non-trivial convergent sequences, and such that for every Boolean algebra $\AAA$ and $X\in\FF$, if $X$ homeomorphically embeds into $St(\AAA)$, then $\AAA$ does not have the Nikodym property.
\end{restatable*}

\medskip

Our results also apply to the Grothendieck property of Boolean algebras, which is closely related to the Nikodym property. This property has its origin in Banach space theory, see \cite{Gro}.

\begin{definition}\label{def:grothendieck}
A Boolean algebra $\AAA$ has the \textit{Grothendieck property} if, for every sequence $\seqnN{\mu_n}$ of finitely additive signed measures on $\AAA$ such that $\sup_{n\in\omega}\|\mu_n\|<\infty$ and $\mu_n(A)\to 0$ for all $A\in\AAA$, we have $\widehat{\mu}_n(B)\to 0$ for every Borel subset $B\sub St(\AAA)$, where $\widehat{\mu}_n$ denotes for each $n\in\omega$ the Radon extension of $\mu_n$ onto the space $St(\AAA)$.
\end{definition}

The connection between the Grothendieck property and the Nikodym property of Boolean algebras has been a subject of an extensive research (see e.g. \cite{Fre}, \cite{Hay} or \cite{Sch}), but it is still not fully understood. It is not easy to find an algebra having exactly one of these properties. Schachermayer \cite{Sch} showed that the algebra $\mathscr{J}$ of Jordan-measurable subsets of the unit interval has the Nikodym property but does not have the Grothendieck property (see also \cite{GraWhe}, \cite{Val}, \cite{SZMin}). On the other hand, Talagrand \cite{Tal} constructed  under the assumption of the Continuum Hypothesis a Boolean algebra with the Grothendieck property and without the Nikodym property, but it is still an open problem if one can find such an example in ZFC (cf. \cite{SZMA}, \cite{GW}).

Marciszewski and Sobota studied in \cite{MS} a property of spaces of the form $N_F$ that is connected with the Grothendieck property in a similar way as the finitely supported Nikodym property is connected with the Nikodym property of Boolean algebras.
We need again to introduce some piece of notation. For a topological space $X$ we denote by $C(X)$ the set of all continuous real-valued functions on $X$ and by $C_b(X)$ the set of all bounded continuous real-valued functions on $X$. 
If $\mu$ is a Borel measure on $X$ and $f\in C(X)$, then we write $\mu(f)=\int_Xfd\mu$.

\begin{definition}
Let $X$ be a Tychonoff space. A sequence $\seqnN{\mu_n}$ of finitely supported measures on $X$ is a \textit{bounded Josefson--Nissenzweig sequence}, or shortly a \textit{BJN-sequence}, if $\|\mu_n\|=1$ for every $n\in\omega$ and $\mu_n(f)\rightarrow 0$ for every $f\in C_b(X)$. We say that $X$ has the \textit{bounded Josefson--Nissenzweig property}, or shortly the \textit{BJNP}, if there is a BJN-sequence of measures on $X$. 

We also say that a free filter $F$ on $\omega$ has the \textit{bounded Josefson--Nissenzweig property}, or shortly the \textit{BJNP}, if the space $N_F$ has the BJNP. We denote by $\BB\JJ\NN\PP$ the class of all ideals on $\omega$ whose dual filters have the BJNP.
\end{definition}

The bounded Josefson--Nissenzweig property was introduced in \cite{KMSZ}, as a tool for investigating the linear structure of the spaces $C_b(X\times Y)$ endowed with the pointwise topology.
It is proved in \cite[Corollary D]{MS} that if $\AAA$ is a Boolean algebra, $F$ is a free filter on $\omega$ that has the BJNP, and $N_F$ homeomorphically embeds into $St(\AAA)$, then $\AAA$ does not have the Grothendieck property.

 By Proposition \ref{nikodym_bjn}, $\AAA\NN$ is a subclass of $\BB\JJ\NN\PP$. As we mentioned earlier, by Theorem \ref{thm_no_maximal} no element of $\AAA\NN$ is maximal with respect to the order $\leq_K$. In contrast, the class $\BB\JJ\NN\PP$ has $\leq_K$-maximal elements, e.g. the asymptotic density zero ideal $\ZZ$ is such an element (see \cite[Theorem C]{MS}). Moreover, Theorem \ref{nikodym_equiv} shows that for a density ideal $\II$ it is equivalent that $\II\notin\AAA\NN$ and that $\II$ is a $\leq_K$-maximal element in $\BB\JJ\NN\PP$. For more details on differences between the classes $\AAA\NN$ and $\BB\JJ\NN\PP$ see Section $6$.
 
\medskip
 
The structure of the paper is as follows. In Section 2 we describe basic notation and terminology used in this paper.
In Section 3 we analyze basic properties of anti-Nikodym sequences of measures on spaces $N_F$. Section 4 is devoted to proving characterizations of the Nikodym property of filters, and studying connections between this property and submeasures on $\omega$. In Section 5 we present several conditions equivalent to $\II\not\in\AAA\NN$ for a density ideal $\II$, and we investigate the \katetov\ order on the class $\AAA\NN$, in particular its cofinal structure. In Section 6 we compare the Nikodym property with the bounded Josefson--Nissenzweig property of filters, and also classes of ideals on $\omega$ related to these properties. In Section 7 we analyze some families of ideals that are contained in the class $\AAA\NN$, which give us some rich substructures in the order $(\AAA\NN, \leq_K)$. Section 8 is devoted to applications of the Nikodym property of filters to Boolean algebras. The last section provides some remarks and open questions.

\subsection*{Acknowledgements}
The author would like to thank Damian Sobota for many inspiring discussions that contributed to the results of this paper, as well as for helpful comments during the preparation of the paper.

\section{Preliminaries}

By $\omega$ we denote the first infinite (countable) cardinal number. We also denote $\omega_+=\omega\sm\{0\}=\{1,2,3,\ldots\}$. For $A,B\sub\omega$ we write $A\sub^*B$ if $A\sm B$ is finite. For every $k,n\in\omega$ such that $k<n$ we set $[k,n]=\{k,k+1,\ldots,n\}$. By $\con$ we denote the cardinality of the real line $\er$, and by $\con^+$---the cardinal successor of $\con$.

If $A$ is a non-empty set, then $\PP(A)$ denotes its power set. A family $F\sub\PP(A)$ is a \textit{filter on }$A$ if $\emptyset\notin F$, $A\in F$ and $F$ is closed under finite intersections and taking supersets. By $Fr$ we denote the \textit{Fr\'echet filter} on $\omega$, i.e. $Fr=\{B\sub\omega\colon\omega\sm B$ is finite$\}$. A filter $F$ on $\omega$ is \textit{free} if $Fr\sub F$. Throughout this article, all filters on $\omega$ are assumed to be free. For a family $S\sub\PP(A)$ we put $S^*=\{X\colon A\sm X\in S\}$---$S^*$ is said to be \textit{dual to $S$}.

A family $\II\sub\PP(\omega)$ is an \textit{ideal on }$\omega$ if $\II^*$ is a free filter on $\omega$. Note that this implies $Fin\sub\II$, where $Fin$ is the ideal of all finite subsets of $\omega$. An ideal $\II$ on $\omega$ is a \textit{P-ideal} if for every sequence $\seqnN{A_n\in\II}$ there is $A\in\II$ such that $A_n\sub^*A$ for every $n\in\omega$. We say that an ideal $\II$ on $\omega$ is \textit{tall} if for every infinite $A\sub\omega$ there exists an infinite $B\sub A$ such that $B\in\II$.

We denote the set of all functions from $\omega$ to $\omega$ by $ \Baire$. We define a preorder $\leq^*$ on $\Baire$ by setting $f\leq^* g$ if $f(n)\leq g(n)$ for all but finitely many $n\in\omega$. There are two cardinal invariants related to this preorder:
\begin{itemize}
\item $\dom$ is the smallest size of a family $\FF\sub\Baire$ such that for every $g\in\Baire$ there is $f\in\FF$ with $g\leq^*f$ (such a family is called \textit{dominating} in $\Baire$);
\item $\unb$ is the smallest size of a family $\FF\sub\Baire$ such that for every $g\in\Baire$ there is $f\in\FF$ with $f\nleq^*g$ (such a family is called \textit{unbounded} in $\Baire$).
\end{itemize}

\noindent The invariants $\dom$ and $\unb$ are examples of the \textit{cardinal characteristics of the continuum}, i.e. cardinal invariants which are between $\aleph_1$ and $\con$. Such an invariants are studied extensively in the field of infinitary combinatorics, see \cite{Bla}.

Let $\AAA$ be a Boolean algebra. $St(\AAA)$ denotes the Stone space of $\AAA$, i.e. the space of all ultrafilters on $\AAA$ endowed with the standard topology. If $A\in\AAA$, then $[A]_{\AAA}$ denotes the clopen set in $St(\AAA)$ corresponding via the Stone duality to $A$. 

A function $\mu\colon\AAA\to\er$ is called a \textit{measure on a Boolean algebra} $\AAA$ if it is a finitely additive function of bounded total variation, that is, 
\[\|\mu\|=\sup\big\{|\mu(A)|+|\mu(B)|\colon A,B\in\AAA, A\land B=0_{\AAA}\big\}<\infty.\]

Let $X$ be a Tychonoff space. By $\clopen(X)$ we denote the family of clopen subsets of $X$, i.e. subsets which are both open and closed. When we say that $\mu$ is a \textit{measure on} $X$, then we mean that $\mu$ is a $\sigma$-additive regular signed measure defined on the Borel $\sigma$-algebra $Bor(X)$ of $X$ which has bounded total variation, that is,
\[\|\mu\|=\sup\big\{|\mu(A)|+|\mu(B)|\colon A,B\in Bor(X), A\cap B=\emptyset\big\}<\infty.\]
 By $|\mu|(\cdot)$ we denote the \textit{variation} of $\mu$; note that $\|\mu\|=|\mu|(X)$. $\supp(\mu)$ denotes the \textit{support} of $\mu$. If $A\in Bor(X)$, then $\mu\restriction A$ is defined by the formula 
$(\mu\restriction A)(B)=\mu(A\cap B)$ for every $B\in Bor(X)$. $\mu$ is a \textit{non-negative measure} if $\mu(A)\geq 0$ for every $A\in Bor(X)$, and it is a \textit{probability measure} if it is non-negative and $\mu(X)=1$.

We define the \textit{Radon extension} $\widehat{\mu}$ of a measure $\mu$ on a Boolean algebra $\AAA$ as the unique measure on $St(\AAA)$ such that $\widehat{\mu}\big([A]_{\AAA}\big)=\mu(A)$ for every $A\in\AAA$.

If $X$ is a Tychonoff space and $x\in X$, then $\delta_x$ denotes the one-point measure on $X$ concentrated at $x$.
We say that a measure $\mu$ on $X$ is \textit{finitely supported} if it is of the form $\mu=\sum_{i=1}^n\alpha_i\cdot\delta_{x_i}$ for some $\alpha_1,\ldots,\alpha_n\in\er\sm\{0\}$ and distinct $x_1,\ldots,x_n\in X$. 
It follows that 
$\supp(\mu)=\{x_1,\ldots,x_n\}$, $\|\mu\|=\sum_{i=1}^n|\alpha_i|$, and that for every $A\in Bor(X)$ we have $\mu\restriction A=\sum_{i\colon x_i\in A}\alpha_i\cdot\delta_{x_i}$ and $|\mu|(A)=\|\mu\restriction A\|=\sum_{i\colon x_i\in A}|\alpha_i|$.

For a non-negative measure $\mu$ supported on a finite set
$X$ we define the following numbers (cf. \cite[Section 1.13]{Far}):
\[\at^+(\mu)=\max \big\{\mu(\{x\})\colon x\in X\big\}, \]
\[\at^-(\mu)=\min \big\{\mu(\{x\})\colon x\in X\big\}. \]

If $\mu$ is a measure on $X$ and $f\in C(X)$, then we set $\mu(f) = \int_{X}fd\mu$ if the integral exists. Note that if $\mu$ is of the form $\sum_{i=1}^n\alpha_i\cdot\delta_{x_i}$, then $\mu(f)=\sum_{i=1}^n\alpha_i\cdot f(x_i)$.

\section{Anti-Nikodym sequences of measures on spaces $N_F$}

In this section we study the main properties of anti-Nikodym sequences of measures on spaces $N_F$. 

We start with a remark about the possible generalization of the finitely supported Nikodym property. One may ask whether we could exchange the condition that $\mu_n(A)\rightarrow 0$ for every $A\in\clopen(X)$ for the condition that $\mu_n(f)\rightarrow 0$ for every $f\in C_b(X)$, and in this way generalize the definition of the finitely supported Nikodym property to non-zero-dimensional spaces. However, this property is not satisfied for any Tychonoff space $(X,\tau)$ which has a weaker compact Hausdorff topology $\tau'\sub\tau$, as it would imply that there is a sequence of measures on $(X, \tau')$ which is weak* convergent to $0$ and unbounded in norm in the dual space $C(X,\tau')^*$  of bounded linear functionals on the Banach space $C(X,\tau')$---that would contradict the Banach--Steinhaus theorem for $C(X,\tau')$.
In particular, for any filter $F$ the space $N_F$ does not satisfy this stronger property, as $F$ contains the Fr\'echet filter $Fr$ and so the topology $\tau_{Fr}$ is weaker than $\tau_F$, and $N_{Fr}$ is compact being homeomorphic to a convergent sequence of real numbers together with its limit point.

\begin{remark} \label{rem_subfilter}
If $F\sub G$ are filters on $\omega$, then an AN-sequence $\seqnN{\mu_n}$ of measures on $N_G$ is also AN-sequence on $N_F$, because the topology of $N_G$ is \textit{finer} than the topology of $N_F$, so $\clopen(N_F)\sub\clopen(N_G)$. The same is true for BJN-sequences of measures.
\end{remark}

The following result shows that we can always modify an anti-Nikodym sequence of measures on a space $N_F$ so that the measures have disjoint supports. It is known that the analogous result holds for BJN-sequences of measures on any Tychonoff space (see \cite[Lemma 5.1.(3)]{MS}).

\begin{proposition}\label{prop_disjoint}
If a filter $F$ on $\omega$ does not have the Nikodym property, then there is an AN-sequence $\seqnN{\mu_n}$ of measures on $N_F$ which is disjointly supported, that is, $\supp(\mu_k)\cap\supp(\mu_l)=\emptyset$ for every $k\neq l\in\omega$. 
\end{proposition}

\begin{proof}
Let $\seqnN{\mu_n}$ be an AN-sequence of measures on $N_F$.

\underline{\textbf{Step 1.}}  We will find an AN-sequence $\seqkN{\theta_k}$ on $N_F$ such that $\supp(\theta_k)\cap\supp(\theta_l)\cap\omega=\emptyset$ for every $k\neq l\in\omega$, i.e. with supports disjoint on $\omega$.

We first define inductively a strictly increasing sequence $\seqkN{n_k}$ of natural numbers and an increasing sequence $\seqkN{A_k}$ of subsets of $N_F$ satisfying $A_k=\big[0,\max\big(\bigcup_{i=0}^{k-1}\supp(\mu_{n_i})\cap\omega\big)\big]$ and $|\mu_{n_k}|(A_k) < 1/k$ for every $k\in\omega$. We start with $n_0=0$ and $A_0=\emptyset$. Assume now that for some $k\in\omega_+$ we have defined natural numbers $n_0<\ldots <n_{k-1}$ and subsets $A_0\sub\ldots\sub A_{k-1}$ of $N_F$ as required.
The set $A_k=\big[0,\max\big(\bigcup_{i=0}^{k-1}\supp(\mu_{n_i})\cap\omega\big)\big]$ is a finite subset of $\omega$.
As $\seqnN{\mu_n}$ is an AN-sequence and singletons are clopen, there exists $n_k>n_{k-1}$ such that $|\mu_{n_k}|(A_k) < 1/k$.

We now define a measure \[\theta_k=\mu_{n_k}\restriction(N_F\sm A_k)\] for every $k\in\omega$. The sequence $\seqkN{\theta_k}$ of measures on $N_F$ will meet the requirements, as:
\begin{itemize}
\item for every pair of natural numbers $k<l$ we have $\supp(\theta_k)\cap\omega\sub A_{k+1}$ and $\supp(\theta_l)\cap A_{k+1}=\emptyset$, by the definition of the sequence $\seqkN{A_k}$,
\item for all $k\in\omega_+$ we have \[\|\theta_k\|\geq \|\mu_{n_k}\| - \|\mu_{n_k}\restriction A_k\|>\|\mu_{n_k}\| - 1/k\textrm{, and so }\|\theta_k\|\rightarrow\infty,\]
\item for all $A\in\clopen(N_F)$ we have $\theta_k(A)\rightarrow 0$, as $\mu_{n_k}(A)\rightarrow 0$ and \[ \big|\theta_k(A) - \mu_{n_k}(A)\big|\leq\|\theta_k - \mu_{n_k}\|=\|\mu_{n_k}\restriction A_k\|<1/k \rightarrow 0. \]
\end{itemize}

\underline{\textbf{Step 2.}}  In Step $1$ we constructed an AN-sequence $\seqnN{\theta_n}$ on $N_F$ with supports disjoint on $\omega$, and we will now ``remove'' the point $p_F$ from the supports. We have two cases.
\begin{enumerate}[a)]
\item $\liminf_{n\to\infty}\big|\theta_n(\{p_F\})\big| < \infty$:

We can find a subsequence $\seqkN{\theta_{n_k}}$ such that $\theta_{n_k}\big(\{p_F\}\big)\to\alpha$ for some $\alpha\in\er$, hence $\seqkNbig{\theta_{n_k}\big(\{p_F\}\big)}$ is a Cauchy sequence. We define a sequence $\seqkN{\nu_k}$ of measures by setting
\[ \nu_k=(\theta_{n_{2k}} - \theta_{n_{2k+1}})\restriction\omega \]
for every $k\in\omega$. Then, $\seqkN{\nu_k}$ is a disjointly supported AN-sequence on $N_F$, as:
\begin{itemize}
\item for every $k\in\omega$ we have\[\supp(\nu_k)=\big(\supp(\theta_{n_{2k}})\cup \supp(\theta_{n_{2k+1}})\big)\cap\omega,\] thus $\supp(\nu_k)\cap\supp(\nu_l)=\emptyset$ for $k\neq l$ as $\theta_n$'s have disjoint supports on $\omega$,
\item we have $\|\nu_k\| = \|\theta_{n_{2k}}\restriction\omega\| + \|\theta_{n_{2k+1}}\restriction\omega\|\rightarrow\infty$, because
$\|\theta_{n_{l}}\restriction\omega\|\rightarrow\infty$ (which follows from $\|\theta_{n_{l}}\|\rightarrow\infty$ and $\|\theta_{n_{l}}\restriction\{p_F\}\|\rightarrow|\alpha|$),
\item  for all $A\in\clopen(N_F)$ we have:
\[ \nu_k(A)=\big(\theta_{n_{2k}} - \theta_{n_{2k+1}}\big)(A\cap\omega)= \big(\theta_{n_{2k}} - \theta_{n_{2k+1}}\big)(A) - \big(\theta_{n_{2k}} - \theta_{n_{2k+1}}\big)(A\cap\{p_F\})\rightarrow 0,\]
as $\seqkN{\theta_{n_k}}$ is an AN-sequence and $\theta_{n_{2k}}(\{p_F\}) - \theta_{n_{2k+1}}(\{p_F\})\rightarrow 0$ holds because $\seqkNbig{\theta_{n_k}\big(\{p_F\}\big)}$ is Cauchy.
\end{itemize}

\item $\lim_{n\to\infty}\big|\theta_n(\{p_F\})\big| = \infty$:

Without losing of generality we may assume that the sequence $\seqnNbig{\big|\theta_n(\{p_F\})\big|}$ is strictly increasing. If we denote
\[\alpha_n=\frac{\theta_{2n}(\{p_F\})}{\theta_{2n+1}(\{p_F\})},\] 
then we get that $|\alpha_n|<1$ for all $n\in\omega$.
We define a sequence $\seqnN{\nu_n}$ of measures by setting
\[ \nu_n=(\theta_{2n} - \alpha_n\cdot\theta_{2n+1})\restriction\omega\] for every $n\in\omega$.
Then, $\seqnN{\nu_n}$ is a disjointly supported AN-sequence on $N_F$, as:
\begin{itemize}
\item for every $n\in\omega$ we have \[\supp(\nu_n)=\big(\supp(\theta_{2n})\cup \supp(\theta_{2n+1})\big)\cap\omega,\] thus $\supp(\nu_k)\cap\supp(\nu_l)=\emptyset$ for $k\neq l$ as $\theta_n$'s have disjoint supports on $\omega$,
\item we have $\|\nu_n\| = \|\theta_{2n}\restriction\omega\| + |\alpha_n|\cdot\|\theta_{2n+1}\restriction\omega\|\rightarrow\infty$, because for every $k\in\omega$ it holds
\[ \|\theta_k\restriction\omega\|\geq|\theta_k(\omega)|=|\theta_k(N_F) - \theta_k(\{p_F\})|\rightarrow\infty,\]
which follows from $\theta_k(N_F)\rightarrow 0$ and the assumption $\big|\theta_k(\{p_F\})\big|\rightarrow\infty$,
\item  for all $A\in\clopen(N_F)$ we have:
\[ \nu_n(A)=\big(\theta_{2n} - \alpha_n\cdot\theta_{2n+1}\big)(A\cap\omega)= \big(\theta_{2n} - \alpha_n\cdot\theta_{2n+1}\big)(A) - \big(\theta_{2n} - \alpha_n\cdot\theta_{2n+1}\big)(A\cap\{p_F\})\rightarrow 0,\]
as $\theta_{2n}(\{p_F\}) - \alpha_n\cdot\theta_{2n+1}(\{p_F\})= 0$, $\seqnN{\theta_n}$ is an AN-sequence and $\seqnN{\alpha_n}$ is bounded.
\end{itemize}
\end{enumerate}
\end{proof}

We will now characterize when a sequence of finitely supported measures on a space $N_F$ is an AN-sequence.

\begin{proposition}\label{prop_characterization}
Let $F$ be a filter on $\omega$ and $\seqnN{\mu_n}$ a sequence of finitely supported measures on $N_F$. Then $\seqnN{\mu_n}$ is an AN-sequence on $N_F$ if and only if the following three conditions simultaneously hold:
\begin{enumerate}[(1)]
\item $\lim_{n\to\infty}\|\mu_n\|=\infty$,
\item $\lim_{n\to\infty}\mu_n(N_F)= 0$,
\item $\lim_{n\to\infty}\big\|\mu_n\restriction(\omega\sm A)\big\|=0$ for every $A\in F$.
\end{enumerate}
\end{proposition}
\begin{proof}
Let $\seqnN{\mu_n}$ be an AN-sequence on $N_F$. Then, $(1)$ and $(2)$ hold by the definition. By Proposition \ref{prop_disjoint}, without losing of generality we may assume that the sequence $\seqnN{\mu_n}$ is disjointly supported. Assume, for the sake of contradiction, that for some $A\in F$ there is a subsequence $\seqkN{\mu_{n_k}}$ such that $\lim_{k\to\infty} \big\|\mu_{n_k}\restriction(\omega\sm A)\big\| = C$ for some $C>0$. We may additionally assume that $\big\|\mu_{n_k}\restriction(\omega\sm A)\big\| \geq C/2$ for every $k\in\omega$. For each $k\in\omega$ there exists $F_k\sub\supp\big(\mu_{n_k}\restriction(\omega\sm A)\big)$ satisfying $|\mu_{n_k}(F_k)| \geq C/4$. Since $\omega\sm A$ is a discrete clopen subspace of $N_F$, the set $F=\bigcup_{k\in\omega}F_k$ is clopen in $N_F$. But $\lim_{k\to\infty}|\mu_{n_k}(F)|\geq C/4$, contradicting the fact that $\seqkN{\mu_{n_k}}$ is an AN-sequence on $N_F$.


Assume now that $\seqnN{\mu_n}$ is a sequence of finitely supported measures on $N_F$ satisfying conditions $(1)$--$(3)$. If $C\in\clopen(N_F)$ is such that $p_{F}\notin C$, then it must be of the form $\omega\sm A$ for some $A\in F$, and so $\mu_n(C)\rightarrow 0$ by $(3)$. On the other hand, if $C\in\clopen(N_F)$ is such that $p_F\in C$, then we already know that $\mu_n(N_F\sm C)\rightarrow 0$, thus $\mu_n(C)=\mu_n(N_F)-\mu_n(N_F\sm C)\rightarrow 0$, by $(2)$.
\end{proof}

\section{The Nikodym property and submeasures on $\omega$} \label{section_ideals}

In this section we prove our main results about filters $F$ without the Nikodym property and their corresponding dual ideals. We will first characterize them using a sequences of non-negative measures on $\omega$, and next using particular ideals on $\omega$ associated with submeasures. Theorems \ref{thm_positive_measures} and \ref{thm_AN_exh}, and their proofs, were strongly inspired by similar characterizations of filters with the BJNP due to Marciszewski and Sobota \cite[Theorem B and Theorem 6.2]{MS}.

\theorempositivemeasures

\begin{proof}
Assume that the filter $F$ does not have the Nikodym property. By Proposition \ref{prop_disjoint}, there is an AN-sequence $\seqnN{\theta_n}$ of measures on $N_F$ with pairwise disjoint supports contained entirely in $\omega$. Let us define a measure $\mu_n=|\theta_n|$ for each $n\in\omega$. It follows that each $\mu_n$ is non-negative and finitely supported, $\mu_n(\omega)=\|\theta_n\|\rightarrow\infty$, and, by Proposition $\ref{prop_characterization}.(3)$, that
 \[ \lim_{n\to\infty}\,\mu_n(\omega\sm A)= \lim_{n\to\infty}\,\big\|\theta_n\restriction(\omega\sm A)\big\|=0\]
for every $A\in F$. Thus, conditions $(1)$--$(3)$ are satisfied.

Assume now that there is a sequence $\seqnN{\mu_n}$ of finitely supported non-negative measures on $N_F$ satisfying conditions $(1)$--$(3)$. For each $n\in\omega$ let us define a measure:
\[\nu_n = \mu_n(\omega)\cdot\delta_{p_F}-\mu_n.\]
By Proposition \ref{prop_characterization}, it is immediate that $\seqnN{\nu_n}$ is an AN-sequence on $N_F$.
\end{proof}
\noindent Note that conditions $(2)$ and $(3)$ in the above theorem imply $\mu_n(A)\to\infty$ for $A\in F$.

Let us recall some standard definitions concerning submeasures on $\omega$. A function $\varphi\colon\PP(\omega)\to[0,\infty]$ is a \textit{submeasure} if $\varphi(\emptyset)=0$, $\varphi(\{n\})<\infty$ for every $n\in\omega$, $\varphi(X)\leq\varphi(Y)$ whenever $X\sub Y$, and $\varphi(X\cup Y)\leq\varphi(X)+\varphi(Y)$ for every $X, Y$. A submeasure $\varphi$ on $\omega$ is \textit{lower semicontinuous (lsc)} if $\varphi(A)=\lim_{n\to\infty}\varphi(A\cap [0,n])$ for every $A\sub\omega$. 
We consider the following ideal associated with an lsc submeasure $\varphi$ on $\omega$:
\[\exh(\varphi)=\big\{A\sub\omega\colon\lim_{n\to\infty} \varphi(A\sm [0,n])=0\big\},\]
called the \textit{exhaustive ideal} of $\varphi$.
It is not difficult to show that $\exh(\varphi)$ is an $F_{\sigma\delta}$ P-ideal for every lsc submeasure $\varphi$. There following result of Solecki characterizes analytic P-ideals on $\omega$ in terms of submeasures.

\begin{theorem}[Solecki \cite{Sol}]
Let $\II$ be an ideal on $\omega$ such that $Fin\sub\II$. Then $\II$ is an analytic P-ideal if and only if there is a lsc submeasure $\varphi$ such that $\II=\exh(\varphi)$.
\end{theorem}

For submeasures $\varphi$ and $\psi$ on $\omega$ we write $\psi\leq\varphi$ if $\psi(A)\leq\varphi(A)$ for every $A\sub\omega$. Following Farah \cite[page 21]{Far}, we call a submeasure $\varphi$ \textit{non-pathological} if for every $A\sub\omega$ we have:
\[ \varphi(A)=\sup \{\mu(A)\colon \mu\textrm{ is a non-negative measure on } \omega \textrm{ such that } \mu\leq\varphi\}.\]
Trivially, every non-negative measure on $\omega$ is non-pathological. 

The family of density submeasures is an important subfamily of non-pathological lsc submeasures. Recall that a submeasure $\varphi$ on $\omega$ is a \textit{density submeasure} if there exists a sequence $\seqnN{\mu_n}$ of finitely supported non-negative measures on $\omega$ with pairwise disjoint supports
such that 
\[\varphi = \sup_{n\in\omega}\mu_n.\]
An ideal $\II$ on $\omega$ is a \textit{density ideal} if there is a density submeasure $\varphi$ such that $\II=\exh(\varphi)$. We say in this situation that $\II$ is \textit{generated by the sequence} $\seqnN{\mu_n}$. Note that $X\in\exh(\varphi)$ if and only if 
$\lim_{n\to\infty}\mu_n(X)=0$.

A prototypical ideal for the class of density ideals is the \textit{(asymptotic) density zero} ideal $\ZZ = \exh(\varphi_d)$, where the \textit{asymptotic density submeasure} $\varphi_d$ is defined for every $A\sub\omega$ by:
\[\varphi_d(A)=\sup_{n\in\omega}\frac{\big|A\cap[2^n, 2^{n+1})\big|}{2^n}.\]

We are now ready to provide a characterization of the Nikodym property of filters in terms of exhaustive ideals.
\theoremANexh

\begin{proof}
$(1)\Rightarrow(2)$ Assume that $F$ does not have the Nikodym property. Let $\seqnN{\mu_n}$ be a sequence of finitely supported non-negative measures on $N_F$ as in Theorem \ref{thm_positive_measures}.
By the remark after the theorem, we may additionally assume that $\mu_n$'s have pairwise disjoint supports such that $\max\big(\supp(\mu_n)\big) < \min\big(\supp(\mu_{n+1})\big)$ for every $n\in\omega$. 
For each $A\sub\omega$ define:
\[\varphi(A)=\sup_{n\in\omega}\mu_n(A).\]

Then $\varphi$ is a density submeasure. We have $\varphi(\omega)=\sup_{n\in\omega}\mu_n(\omega)=\infty$ by condition $(2)$ from Theorem \ref{thm_positive_measures}.
It is fairly easy to see that $F\sub\exh(\varphi)^*$ (cf. \cite[Theorem 6.2]{MS}).

$(2)\Rightarrow(3)$ Obvious.

$(3)\Rightarrow(1)$ Assume that there is a non-pathological lsc submeasure $\varphi$ such that $\varphi(\omega)=\infty$ and $F\sub\exh(\varphi)^*$. By Remark \ref{rem_subfilter}, it is enough to prove that $\exh(\varphi)^*$ does not have the Nikodym property.

As $\varphi$ is finite on finite sets and $\varphi(\omega)=\infty$, by the subadditivity of $\varphi$ for every $n\in\omega$ we get: 
\begin{equation}\label{infinite} \tag{$\ast$}
\varphi\big(\omega\sm [0,n]\big)=\infty.
\end{equation}
Let $n_0 = 0$. Since $\varphi$ is lower semi-continuous, there exists $n_1>n_0$ such that $\varphi([n_0,n_1])>2$. By (\ref{infinite}) we have $\varphi(\omega\sm [0,n_1])=\infty$, so again there is $n_2>n_1$ such that  $\varphi([n_1,n_2])>3$. We continue in this way until we get a strictly increasing sequence $\seqkN{n_k}$ of natural numbers
satisfying for every $k\in\omega$ the inequality
\[\varphi\big([n_k,n_{k+1}]\big)>k+1.\]

The submeasure $\varphi$ is non-pathological, so for each $k\in\omega$ there exists a non-negative measure $\mu_k$ on $\omega$ such that $\mu_k\leq\varphi$, $\supp(\mu_k)\sub[n_k,n_{k+1}]$, and
\[\mu_k\big([n_k,n_{k+1}]\big)>k.\]
Measures $\mu_k$ are finitely supported, and $\lim_{k\to\infty}\mu_k(\omega)\geq\lim_{k\to\infty}\mu_k\big([n_k,n_{k+1}]\big)=\infty$.

We claim that the sequence $\seqkN{\mu_k}$ satisfies the equality $\lim_{k\to\infty} \mu_k(A) = 0$ for every $A\in\exh(\varphi)$, and thus, by Theorem \ref{thm_positive_measures}, the filter $\exh(\varphi)^*$ does not have the Nikodym property. Let $A\in\exh(\varphi)$ and $\eps>0$.
Since $\lim_{n\to\infty}\varphi\big(A\sm [0,n]\big)=0$, there is $M\in\omega$ such that $\varphi\big(A\sm [0,M]\big)<\eps$. Let $k\in\omega$ be such that $n_k>M$. It holds:
\[\mu_k(A)=\mu_k\big(A\cap[n_k, n_{k+1}]\big)\leq\varphi\big(A\cap[n_k, n_{k+1}]\big)\leq \varphi\big(A\sm [0,M]\big)<\eps,\]
hence $\lim_{k\to\infty} \mu_k(A) = 0$.
\end{proof}

Recall that by classical theorem of Sierpiński (see \cite[Theorem 4.1.1]{BJ}), every ideal of the form $\exh(\varphi)$ for some lsc submeasure $\varphi$ is meager and of measure zero (as it is a Borel subset of $\cantor$), and every free ultrafilter on $\omega$ is non-meager. Therefore, we get the following corollaries.

\begin{corollary}\label{ultrafilter}
If a filter $F$ on $\omega$ does not have the Nikodym property, then $F$ is meager and of measure zero. In particular, if $F$ is a free ultrafilter on $\omega$, then $F$ has the Nikodym property.
\end{corollary}

It appears that the notion of a totally bounded ideal is closely related to the Nikodym property of filters. 

\begin{definition}[\hernandez\textrm{ and }\hrusak\,\cite{HHH}]
An analytic P-ideal $\II$ on $\omega$ is said to be \textit{totally bounded} if whenever $\varphi$ is a lower semicontinuous submeasure on $\omega$ for which $\II=\exh(\varphi)$, then $\varphi(\omega)<\infty$.
\end{definition}

A typical example of a totally bounded ideal is $\ZZ$ (\cite[Proposition 3.18]{HHH}), whereas, e.g., no summable ideal is totally bounded (see Section $7.1$). We will study totally bounded ideals further in the context of density ideals in the next section.

\begin{corollary}\label{tot_bounded}
Let $\II=\exh(\varphi)$ for an lsc submeasure $\varphi$ on $\omega$.
If $\II$ is totally bounded, then $\II^*$ has the Nikodym property.
\end{corollary}
\begin{proof}

If $\II^*$ does not have the Nikodym property, then by Theorem \ref{thm_AN_exh} there is an lsc submeasure $\psi$ on $\omega$ such that $\psi(\omega)=\infty$ and $\exh(\varphi)\sub\exh(\psi)$.
We define an lsc submeasure on $\omega$ by $\varphi'=\max(\psi, \varphi)$. Of course, we have $\varphi'(\omega)=\infty$ and $\exh(\varphi')\sub\exh(\varphi)$. Now, if $A\in\exh(\varphi)$, then also $A\in\exh(\psi)$, so 
\[\lim_{n\to\infty}\varphi'(A\sm [0,n])=\max\Big(\lim_{n\to\infty}\varphi(A\sm [0,n]), \lim_{n\to\infty}\psi(A\sm [0,n])\Big)=0, \]
and hence $\exh(\varphi)\sub\exh(\varphi')$. Thus we have $\exh(\varphi)=\exh(\varphi')$, and so $\II$ is not totally bounded.

\end{proof}

\begin{remark}\label{nik_tot_bounded}
Let us note that the converse to Corollary \ref{tot_bounded} does not hold in general. The counterexamples were given in \cite[Theorem 4.12]{FT19} (with the ideal defined in \cite{FS10}) and in \cite[Example 6.15]{MS}. Both of them are of the form $\II=\exh(\varphi)=\finI(\varphi)$ for some lsc submeasure $\varphi$ on $\omega$ such that $\II^*$ does not have the BJNP. By Proposition \ref{nikodym_bjn}, either filter $\II^*$ has the Nikodym property, and $\varphi(\omega)=\infty$ because $\omega\notin\finI(\varphi)$, thus $\II$ is not totally bounded.
\end{remark}

\section{The Nikodym property and the \katetov\ order} \label{section_katetov}

We will focus in this section on the structure of the \katetov\ order (and its variant) on the ideals connected with the Nikodym property. 

Let $\II$ and $\JJ$ be ideals on $\omega$. We say that $\II$ is \textit{\katetov\ below } $\JJ$, which we denote by $\II\leq_K\JJ$, if there is a function $f\colon\omega\to\omega$ such that $f^{-1}[A]\in\JJ$ for all $A\in\II$. We call such a function $f$ a \textit{\katetov\ reduction}.
Similarly, we say that $\II$ is \textit{\katetov--Blass below} $\JJ$, which we denote by $\II\leq_{KB}\JJ$, if there is a finite-to-one function $f\colon\omega\to\omega$ such that $f^{-1}[A]\in\JJ$ for all $A\in\II$, and we call such a function $f$ a \textit{\katetov--Blass reduction}. 
We write $\II<_{K}\JJ$ if $\II\leq_K\JJ$ and $\JJ\nleq_K\II$. We say that $\II$ and $\JJ$ are \textit{\katetov\ equivalent}, which we denote by $\II\equiv_K\JJ$, if $\II\leq_K\JJ$ and $\JJ\leq_K\II$. Trivially, $\II\sub\JJ$ implies $\II\leq_{KB}\JJ$, with the identity function being a \katetov--Blass reduction. 
We say that $\II$ is \textit{isomorphic to} $\JJ$, which we denote by $\II\cong\JJ$, if there is a bijection $f\colon\omega\to\omega$ such that $A\in\II$ if and only if $f[A]\in\JJ$ for all $A\sub\omega$. It is immediate that $\II\cong\JJ$ implies $\II\equiv_K\JJ$.
We define the \katetov\ order for filters on $\omega$ in the natural way via dual ideals.

The following lemma is folklore.

\begin{lemma}\label{katetov_blass}
Let $\II$ and $\JJ$ be ideals on $\omega$. If $\II\leq_K\JJ$ and $\JJ$ is a P-ideal, then $\II\leq_{KB}\JJ$.
\end{lemma}
\begin{proof}\footnote{This argument was presented to us by Jacek Tryba in personal communication.}
Let $f\colon\omega\to\omega$ be a \katetov\ reduction witnessing $\II\leq_K\JJ$. For every $n\in\omega$ the set $A_n=f^{-1}(\{n\})$ belongs to the ideal $\JJ$. By the assumption about $\JJ$, there exists $A\in\JJ$ such that $A_n\sub^* A$ for every $n\in\omega$. We define $g\colon\omega\to\omega$ by:
\[ g(n) = \begin{cases}
	n, & \text{if } n\in A, \\
	f(n), & \text{if } n\notin A.
\end{cases} \]
The function $f$ is finite-to-one on $\omega\sm A$, as $(\omega\sm A)\cap A_n$ is finite for every $n\in\omega$, thus $g$ is finite-to-one on $\omega$. The function $g$ is a \katetov--Blass reduction, because for any $X\in\II$ we have $g^{-1}[X]\sub f^{-1}[X]\cup A\in\JJ$, hence also $g^{-1}[X]\in\JJ$.
\end{proof}

First, we show how the \katetov\ order transfers the negation of the Nikodym property between filters (this strengthens Remark \ref{rem_subfilter}).

\begin{proposition}\label{katetov_nikodym}
If $F\leq_K G$ and $G$ does not have the Nikodym property, then also $F$ does not have the Nikodym property.
\end{proposition}
\begin{proof}
Let $f\colon\omega\to\omega$ be a witness for $F\leq_K G$ and let $\seqnN{\mu_n}$ be a sequence of finitely supported non-negative measures on $N_G$ as in Theorem \ref{thm_positive_measures}. For each $n\in\omega$ define a measure:
\[\mu'_n = \sum_{x\in\supp(\mu_n)} \mu_n(\{x\})\cdot\delta_{f(x)}. \]
Then, we have $\|\mu_n'\|\to\infty$, $\supp(\mu_n')\sub\omega$, and $\mu_n'\geq 0$. We also claim that $\lim_{n\to\infty}\mu_n'(\omega\sm A)=0$ holds for any $A\in F$, and thus, by Theorem \ref{thm_positive_measures}, the filter $F$ does not have the Nikodym property. Indeed, if there was $A\in F$ such that $\limsup_{n\to\infty}\mu_n'(\omega\sm A)>0$, then $f^{-1}[A]\in G$ (as $f$ is a \katetov\ reduction) and 
$\lim\sup_{n\to\infty}\mu_n\big(\omega\sm f^{-1}[A]\big)>0$, which would contradict condition $(3)$ from Theorem \ref{thm_positive_measures}.
\end{proof}

Proposition \ref{katetov_nikodym} yields the following variant of Theorem \ref{thm_AN_exh}.

\begin{theorem} \label{thm_nikodym_exh}
Let $F$ be a filter on $\omega$. Then, the following are equivalent:
\begin{enumerate}[(1)]
\item $F$ does not have the Nikodym property;
\item there is a density submeasure $\varphi$ on $\omega$ such that $\varphi(\omega)=\infty$ and $F\leq_K\exh(\varphi)^*$;
\item there is a non-pathological lsc submeasure $\varphi$ on $\omega$ such that $\varphi(\omega)=\infty$ and $F\leq_K\exh(\varphi)^*$.
\end{enumerate}
\end{theorem}
\begin{proof}
$(1)\Rightarrow(2)$ Let $F$ be a filter on $\omega$ without the Nikodym property. By Theorem \ref{thm_AN_exh}, there is a density submeasure $\varphi$ on $\omega$ such that $\varphi(\omega)=\infty$ and $F\sub\exh(\varphi)^*$, and so $F\leq_K\exh(\varphi)^*$.

$(2)\Rightarrow(3)$ Obvious.

$(3)\Rightarrow(1)$ Let $\varphi$ be a non-pathological lsc submeasure on $\omega$ such that $\varphi(\omega)=\infty$ and $F\leq_K\exh(\varphi)^*$. By Theorem \ref{thm_AN_exh}, the filter $\exh(\varphi)^*$ does not have the Nikodym property, thus by Proposition \ref{katetov_nikodym} the filter $F$  does not have the Nikodym property. 
\end{proof}

\corollaryzet
\begin{proof}
If $F$ does not have the Nikodym property, then by Theorem \ref{thm_nikodym_exh} there exists a density ideal $\II$ such that $F^*\leq_K\II$.
By \cite[Proposition 3.6]{HHH} we have $\II\leq_K\ZZ$, thus $F^*\leq_K\ZZ$.
\end{proof}

We will now prove that, in the case of density ideals, several known properties are equivalent with the Nikodym property. For the definition of \erdos--Ulam ideal see \cite[Section 1.13]{Far}, for the Bolzano--Weierstrass property see \cite[Section 2]{Gda}, and for a splitting family in a Boolean algebra see \cite[page 589]{HHH}.
The following theorem extends \cite[Theorem 3.16]{Try}\footnote{We omit in $(3)$ the condition $\limsup_{n\in\omega}(\omega)>0$, as the definition of an ideal which we use in this paper already implies this condition.}.

\theoremnikodymequiv

\begin{proof}
The equivalence of conditions $(1)$, $(2)$, and $(3)$ was already shown in \cite[Theorem 3.16]{Try}. 

$(2)\Leftrightarrow(4)$ By \cite[Theorem 5.1]{Gda}. 

$(4)\Rightarrow(5)$ By \cite[Lemma 3.17]{HHH}. 

$(5)\Rightarrow(7)$ By Corollary \ref{tot_bounded}. 

$(1)\Rightarrow(6)$ If there exists an \erdos--Ulam ideal $\JJ$ such that $\II\cong\JJ$, then in particular $\II\equiv_K\JJ$. By \cite[Lemma 1.13.10]{Far} we have $\JJ\equiv_K\ZZ$, thus also $\II\equiv_K\ZZ$.

$(6)\Rightarrow(7)$ Assume, for the sake of contradiction, that $\II\equiv_K\ZZ$ and $\II^*$ does not have the Nikodym property. By Proposition \ref{katetov_nikodym} we get that $\ZZ$ does not have the Nikodym property, but $\ZZ$ is totally bounded by \cite[Proposition 3.18]{HHH}, which contradicts Corollary \ref{tot_bounded}.

$(7)\Rightarrow(3)$ Assume that condition $a)$ or condition $b)$ from $(3)$ fails. If $\sup_{n\in\omega}\|\mu_n\| = \infty$, then for the density submeasure $\varphi=\sup_{n\in\omega}\mu_n$ we have $\II=\exh(\varphi)$ and $\varphi(\omega)=\infty$, hence by
Theorem \ref{thm_AN_exh} the filter $\II^*$ does not have the Nikodym property. If $\lim\sup_{n\to\infty} \at^+(\mu_n) > 0$, then $\lim\sup_{n\to\infty} \varphi(\{n\}) > 0$, and so by \cite[Lemma 1.4]{HHH} the ideal $\II$ is not tall. It is well-known that this implies that $\II\equiv_K Fin$, and $Fin^*=Fr$ does not have the Nikodym property by the discussion in  Introduction about non-trivial convergent sequences. Thus, by Proposition \ref{katetov_nikodym}, the filter $\II^*$ does not have the Nikodym property.
In both cases we got a contradiction, hence $(3)$ is satisfied.
\end{proof}

\begin{corollary}\label{EU_Nikodym}
Every \erdos--Ulam ideal is totally bounded and has the Nikodym property.
\end{corollary}

\noindent By Corollary \ref{zz} and Theorem \ref{nikodym_equiv} we get the following corollary.

\begin{corollary} \label{equiv_density}
If $\II$ is a density ideal, then $\II\in\AAA\NN$ if and only if $I<_K\ZZ$.
\end{corollary}

Let us note that the equivalence from the previous corollary does not hold even for ideals of the form $\II=\exh(\varphi)$ for a non-pathological submeasure $\varphi$. A counterexample is given in Section $6$.

\subsection{Cofinal structure of $(\AAA\NN,\leq_K)$}
Now, we will define an operator $\Phi$ from $\Baire$ to $\AAA\NN$ such that its image is $\leq_{K}$-cofinal in $\AAA\NN$ and which preserves dominating families. This will enable us to obtain some properties of the order $(\AAA\NN, \leq_K)$.

Let $f\in\Baire$ be such that $f(n)>0$ for every $n\in\omega_+$.
We define a sequence $\seqnNP{\mu_n^f}$ of measures satisfying: 
\begin{itemize}
\item $\seqnN{I_n}$ is an interval partition of $\omega$, $\min(I_1)=0$ and for every $n\in\omega_+$ we have $\max(I_n)+1=\min(I_{n+1})$,
\item the measure $\mu_n^f$ is supported on the interval $I_n$,
\item $\|\mu_n^f\|=n$,
\item $\at^+(\mu_n^f)=\at^-(\mu_n^f)=1/f(n)$. 
\end{itemize}
Finally, let 
\[\Phi(f)=\exh\Big(\sup_{n\in\omega_+}\mu_n^f\Big).\]

\noindent Then, $\Phi(f)$ is a density ideal.

The following theorem presents the most important properties of the operator $\Phi$.

\begin{theorem}\label{thm_katetov}
\leavevmode
\begin{enumerate}
\item For every $\II\in\AAA\NN$ there exists $f\in\Baire$ such that $\II\leq_K\Phi(f)$. 
\item For every $g,h\in\Baire$, if $2n^2\cdot g(n)\leq h(n)$ for all but finitely many $n\in\omega$, then $\Phi(g)\leq_K\Phi(h)$.
\end{enumerate}
\end{theorem}

For the proof of Theorem \ref{thm_katetov} we will need the following simple lemma.

\begin{lemma}\label{lem_katetov}
Let $\lambda$ be a non-negative measure on a finite non-empty set $A$ and let $\eps > 0$. If $\mu$ is
a non-negative measure on a finite set $B$ such that $\mu(B)=\lambda(A)$ and $\at^+(\mu) \leq \eps/(2|A|)$, then there exists a function
$f\colon B\to A$ such that for every $C\sub A$ we have:
\[ \big| \lambda(C) - \mu\big(f^{-1}[C]\big) \big| \leq \eps.\]
\end{lemma}
\begin{proof}
As $\at^+(\mu) \leq \eps/(2|A|)$, we can construct a family $\langle X_a\colon a\in A\rangle$ of pairwise disjoint subsets of $B$ satisfying
\[ \lambda(\{a\}) - \eps/(2|A|) < \mu(X_a) \leq \lambda(\{a\}) \]
for every $a\in A$. Let us define:
\[ Y = B \sm\bigcup\{X_a\colon a\in A\}, \]
and note that above conditions together with $\mu(B)=\lambda(A)$ imply that $\mu(Y) < \eps/2$.

Fix any $a_0\in A$ and define the function $f\colon B\to A$ in the following way:
\[ f(b) = \begin{cases}
	a, & \text{if } b\in X_a, \\
	a_0, & \text{if } b\in Y.
\end{cases} \]
It is not difficult to check that $f$ has the desired property. 
\end{proof}

The next lemma provides a technical condition sufficient for the existence of a {\katetov } reduction between density ideals.

\begin{lemma}\label{katetov_max}
If $\seqnNP{\mu_n}$ and $\seqnNP{\lambda_n}$ are disjointly supported sequences of measures on $\omega$ with finite supports such that $\|\mu_n\|=\|\lambda_n\|=n$ for every $n\in\omega_+$, and $\at^+(\mu_n)\leq \at^-(\lambda_n)/2n^2$ is satisfied for all but finitely many $n\in\omega_+$, then we have:
\[\exh\Big(\sup_{n\in\omega_+}\lambda_n\Big)\leq_{K} \exh\Big(\sup_{n\in\omega_+}\mu_n\Big). \]
\end{lemma}
\begin{proof}
For each $n\in\omega_+$ let $A_n=\supp(\lambda_n)$. We have $|A_n|\cdot \at^-(\lambda_n)\leq\|\lambda_n\|=n$, thus $\at^-(\lambda_n)\leq n/|A_n|$. Let $N\in\omega_+$ be such that $\at^+(\mu_n)\leq \at^-(\lambda_n)/2n^2$ for every $n\geq N$. Then, for every $n\geq N$ we get:
\[\at^+(\mu_n)\leq \at^-(\lambda_n)/2n^2 \leq\frac{n}{2n^2\cdot|A_n|}=\frac{1}{n}\cdot\frac{1}{2|A_n|}.\]

Therefore, for $n\geq N$ the assumptions of Lemma \ref{lem_katetov} are satisfied with $\lambda=\lambda_n$ on $A=A_n=\supp(\lambda_n)$ and $\mu=\mu_n$ on $B=\supp(\mu_n)$, and $\eps=1/n$. Thus, for every $n\geq N$ there exists a function $f_n\colon\supp(\mu_n)\to A_n$ such that for every $C\sub A_n$ we have:
\begin{equation}\label{measures_close}\tag{$\ast$}
\big| \lambda_n(C) - \mu_n\big(f_n^{-1}[C]\big) \big| \leq 1/n.
\end{equation}

Now we define the function $f\colon\omega\to\omega$ by the formula
\[ f(l) = \begin{cases}
  f_n(l),  & \text{ if } l\in\supp(\mu_n) \text{ for some } n\geq N, \\
  0, & \text{ otherwise,}
\end{cases} \]
for every $l\in\omega$. We will show that $f$ is a witness for $\exh\big(\sup_{n\in\omega_+}\lambda_n\big)\leq_{K} \exh\big(\sup_{n\in\omega_+}\mu_n\big)$, i.e. $f^{-1}[X]\in\exh\big(\sup_{n\in\omega_+}\mu_n\big)$ for any $X\in\exh\big(\sup_{n\in\omega_+}\lambda_n\big)$.
Let us fix $X\in\exh\big(\sup_{n\in\omega_+}\lambda_n\big)$ and $p\geq 1$. There is $m\geq\max(N,p)$ such that for every $n\geq m$ we have $\lambda_n(X)<1/p$.

Then, by (\ref{measures_close}) and the definition of $f$, for every $n\geq m$ we get:
\[\mu_n\big(f^{-1}[X]\big)\leq 2/p, \]
which shows that $\lim_{n\to\infty}\mu_n\big(f^{-1}[X]\big)=0$ and hence that $f^{-1}[X]\in\exh\big(\sup_{n\in\omega_+}\mu_n\big)$.
\end{proof}

\noindent We are in the position to prove Theorem \ref{thm_katetov}.

\begin{proof}[Proof of Theorem \ref{thm_katetov}]
$(1)$. Let $\II\in\AAA\NN$. By Theorem \ref{thm_AN_exh}, there is a density ideal $\exh\big(\sup_{n\in\omega}\lambda_n\big)$ for some sequence $\seqnN{\lambda_n}$ of finitely supported measures with disjoint supports such that $\|\lambda_n\|\to\infty$ and $\II\sub\exh\big(\sup_{n\in\omega}\lambda_n\big)$. Without losing of generality we may assume that $\|\lambda_n\|\geq n+1$ for all $n\in\omega$, as by taking a subsequence of the sequence $\seqnN{\lambda_n}$ we can only enlarge the ideal $\exh\big(\sup_{n\in\omega}\lambda_n\big)$.

We define another sequence $\seqnNP{\nu_n}$ of measures by setting \[\nu_n(A)=\lambda_{n-1}(A)\cdot n/\|\lambda_{n-1}\|\] for each $A\sub\omega$ and $n\in\omega_+$. It follows that $\|\nu_n\|=n$ for every $n\in\omega_+$, and $\II\sub\exh\big(\sup_{n\in\omega_+}\nu_n\big)$.
Indeed, if $A\in\II$ then $\lambda_n(A)\to 0$, and so \[\nu_n(A)=\lambda_{n-1}(A)\cdot n/\|\lambda_{n-1}\|\to 0,\] since $n/\|\lambda_{n-1}\|\leq 1$ for every $n\in\omega$. 

 Let us define $f\in\Baire$ by setting $f(n)=2n^2\cdot \lceil 1/\at^-(\nu_n)\rceil$ for each $n\in\omega_+$ and $f(0)=0$. 
Then, for the sequence of measures $\seqnNP{\mu_n^f}$ generating the density ideal $\Phi(f)$ we have $\at^+(\mu_n^f)\leq \at^-(\nu_n)/2n^2$  for every $n\in\omega_+$, and so by Lemma \ref{katetov_max} we get $\exh\big(\sup_{n\in\omega_+}\nu_n\big)\leq_{K}\Phi(f)$, thus also
$\II\leq_{K}\Phi(f)$.

$(2)$. Let $g,h\in\Baire$ be such that $2n^2\cdot g(n)\leq h(n)$ for all but finitely many $n\in\omega$. It follows immediately from Lemma \ref{katetov_max} that $\Phi(g)\leq_K\Phi(h)$.
\end{proof}

Theorem \ref{thm_katetov} has several important consequences. The next corollary shows that $\Phi$ preserves dominating families.

\begin{corollary}\label{an_domin}
If $\FF$ is a dominating family in $\Baire$, then $\Phi[\FF]$ is $\leq_K$-dominating in $\AAA\NN$, i.e. for any $\II\in\AAA\NN$ there is $f\in\FF$ such that $\II\leq_K\Phi(f)$.
\end{corollary}
\begin{proof}
Let us fix a dominating family $\FF$ in $\Baire$ and let $\II\in\AAA\NN$. By Theorem $\ref{thm_katetov}.(1)$ there is $f\in\Baire$ such that $\II\leq_K\Phi(f)$. As $\FF$ is dominating, there exists $g\in\FF$ such that $g(n)\geq 2n^2\cdot f(n)$ for all but finitely many $n\in\omega$. By Theorem $\ref{thm_katetov}.(2)$ we get $\Phi(f)\leq_K\Phi(g)$, and so $\II\leq_K\Phi(g)$.
\end{proof}

The following theorem follows immediately from Corollary \ref{an_domin}, as by definition there is a dominating family $\FF\sub\Baire$ of size $\dom$.

\corollarynikodymdominating

\begin{corollary}\label{bound}
Every subfamily of $\AAA\NN$ of size less than $\unb$ is bounded in the order $(\AAA\NN,\leq_K)$.
\end{corollary}
\begin{proof}
Let $\{\II_{\alpha}\colon\alpha < \kappa\}$ be a subfamily of $\AAA\NN$ for some $\kappa<\unb$. By Theorem $\ref{thm_katetov}.(1)$, for every $\alpha<\kappa$ there is $f_{\alpha}\in\Baire$ such that $\II_{\alpha}\leq_K\Phi(f_{\alpha})$. As $\kappa<\unb$, there exists $g\in\Baire$ such that for every $\alpha<\kappa$ we have
 $g(n)\geq 2n^2\cdot f_{\alpha}(n)$ for all but finitely many $n\in\omega$, so by Theorem $\ref{thm_katetov}.(2)$ we get $\II_{\alpha}\leq_K\Phi(g)$ for all $\alpha<\kappa$.
\end{proof}

We also have the following result, which is true in fact for all families of tall ideals which are downwards-closed with respect to $\leq_K$.

\begin{corollary}\label{decreasing}
The order $(\AAA\NN,\leq_K)$ contains decreasing chains of cardinality $\con^+$ and antichains of size $\con$.
\end{corollary}
\begin{proof}
By \cite[Theorem 4.2]{HruSur} and Proposition \ref{katetov_nikodym}.
\end{proof}

In Section $7$ we will show that such a $\leq_K$-antichain of size $\con$ in $\AAA\NN$ can consist of summable ideals, or of density ideals. Finally, we prove that the order $(\AAA\NN, \leq_K)$ does not have any maximal elements. 

\theoremnomaximal

\begin{proof}
Let $\II\in\AAA\NN$. By Theorem $\ref{thm_katetov}.(1)$ there is $f\in\Baire$ satisfying $\II\leq_K\Phi(f)$. Moreover, by Theorem $\ref{thm_katetov}.(2)$ we may assume that $f$ is sufficiently big so that both $f(n)\geq n^4$ and $f(n)\geq\sum_{i<n}f(i)$ hold for every $n\in\omega$.
Define $g\in\Baire$ by setting $g(n)=n\cdot f(f(n))$ for each $n\in\omega$. We will show that $\Phi(f)<_K\Phi(g)$, and so $\II<_K\Phi(g)$.

First, we have $g(n)\geq n\cdot\big(f(n)\big)^4$ for every $n\in\omega$, by the assumption that $f(m)\geq m^4$ for $m=f(n)$. Next, for all $n\geq 2$ we have 
\[n\cdot\big(f(n)\big)^4\geq 2n^2\cdot f(n) \textrm{, as } \big(f(n)\big)^3\geq 2n.\] Therefore, by Theorem $\ref{thm_katetov}.(2)$, we get $\Phi(f)\leq_K\Phi(g)$.

Let $\seqnNP{\lambda_n}$ and $\seqnNP{\mu_n}$ be the sequences of measures generating the ideals $\Phi(g)$ and $\Phi(f)$, respectively, as described in the definition of the operator $\Phi$. Let $A_n=\supp(\lambda_n)$ and $B_n=\supp(\mu_n)$ for each $n\in\omega_+$, so that $|A_n|=n^2\cdot f(f(n))$ and $|B_n|=n\cdot f(n)$. We have $\lambda_n(x)=1/\big(n\cdot f(f(n))\big)$ for every $x\in A_n$ and $\mu_n(x)=1/f(n)$ for every $x\in B_n$.

To show that $\Phi(g)\nleq_K\Phi(f)$, assume for the sake of contradiction that there is a \katetov\ reduction $\varphi\colon\omega\to\omega$  witnessing $\Phi(g)\leq_K\Phi(f)$. Since every density ideal is a P-ideal, by Lemma \ref{katetov_blass} we may additionally assume that $\varphi$ is finite-to-one. It is enough to find such a subset $X\sub\omega$ that \[X\notin\exh\big(\sup_{n\in\omega_+}\mu_n\big)=\Phi(f) \textrm{ and } \varphi[X]\in\exh\big(\sup_{n\in\omega_+}\lambda_n\big)=\Phi(g)\] (as $X\sub\varphi^{-1}\big[\varphi[X]\big]$, this will imply $\varphi^{-1}\big[\varphi[X]\big]\notin\exh\big(\sup_{n\in\omega_+}\mu_n\big)$, contradicting that $\varphi$ is a \katetov\ reduction). We proceed in two cases.

\medskip
\underline{\textbf{Case 1.}} There is a strictly increasing sequence $\seqkN{n_k}$ in $\omega_+$ such that for each $k\in\omega$ there exists $F_k\sub B_{n_k}$ satisfying $|F_k|= f(n_k)$ and $\varphi[F_k]\sub\bigcup\big\{A_i\colon f(i)\geq n_k\big\}$.

Let $X=\bigcup_{k\in\omega} F_k$. Since for every $k\in\omega$ we have $\mu_{n_k}(X) = 1$, it follows that $X\notin\exh\big(\sup_{n\in\omega_+}\mu_n\big)$. Next, for every $m\in\omega_+$ we have 
\[\varphi[X]\cap A_m=\bigcup_{k\in\omega}\varphi[F_k]\cap A_m\sub\bigcup\big\{\varphi[F_k]\colon n_k\leq f(m)\big\},\] 
and so
\[ \big|\varphi[X]\cap A_m\big|\leq \sum\Big\{\big|\varphi[F_k]\big|\colon n_k\leq f(m)\Big\} \leq \sum_{i\leq f(m)} f(i) = f(f(m)) + \sum_{i < f(m)} f(i)  \leq 2\cdot f(f(m)), \]
where the last inequality follows from $f(f(m))\geq\sum_{i<f(m)}f(i)$. Therefore, for every $m\in\omega_+$ we have:
\[\lambda_m\big(\varphi[X]\big)\leq \frac{2\cdot f(f(m))}{m\cdot f(f(m))}=\frac{2}{m},\]
which shows that $\lim_{m\to\infty}\lambda_m\big(\varphi[X]\big)=0$, and hence that $\varphi[X]\in\exh\big(\sup_{n\in\omega_+}\lambda_n\big)$.

\medskip
\underline{\textbf{Case 2.}} If we are not in Case $1$, then there exists $N\in\omega_+$ such that for each $n\geq N$ there is no subset $F_n\sub B_n$ satisfying $|F_n|= f(n)$ (i.e. $\mu_n(F_n)=1$) and $\varphi[F_n]\sub\bigcup\big\{A_i\colon f(i)\geq n\big\}$. Then, for every $n\geq N$ there must exist $F_n\sub B_n$ satisfying $\mu_n(F_n)= n-1$ and $\varphi[F_n]\sub\bigcup\big\{A_i\colon f(i) < n\big\}$. Without losing generality we may assume that $N\geq 3$.

First, let us take any $n\geq N$. As $f(i)\geq i^4$ for every $i\in\omega$, we have \[\big|\{i\in\omega_+\colon f(i)<n\}\big|< \sqrt[4]{n} \leq \sqrt{n-1}.\] Thus, from $F_n\sub\bigcup\big\{\varphi^{-1}[A_i]\colon f(i) < n\big\}$ it follows that there exists such a number $i(n)\in\omega_+$ that $f(i(n))<n$ and
\[\mu_n\big(F_n\cap\varphi^{-1}[A_{i(n)}]\big)\geq (n-1)\big/\sqrt{n-1} = \sqrt{n-1}. \]
By $f(i(n))<n$ and $f(i(n))\geq i(n)^4$, it follows that $n\geq i(n)^4 + 1$ and so $\sqrt{n-1}\geq i(n)^2$.
For each $n\geq N$ let us define \[E_n=F_n\cap\varphi^{-1}[A_{i(n)}].\] 
Then, $E_n\sub B_n$, $\mu_n(E_n)=\sqrt{n-1}$, and $\varphi[E_n]\sub A_{i(n)}$. 

For each $n\in\omega_+$ let $A_n=\bigcup_{j=1}^{n^2}C_j^n$ be a division of $A_n$ into $n^2$ disjoint parts of size $f(f(n))$, i.e. $\lambda_n(C_j^n)=1/n$ for every $j=1,\ldots,n^2$ and $C_j^n\cap C_i^n=\emptyset$ for $i\neq j$.  As we have 
\[ E_n\sub\bigcup_{j=1}^{n^2}\varphi^{-1}\Big[C_j^{i(n)}\Big], \]
for every 
$n\geq N$ there exists $j(n)\in\{1,\ldots,n^2\}$ such that
\[\mu_n\Big(E_n\cap\varphi^{-1}\Big[C_{j(n)}^{i(n)}\Big]\Big)\geq \sqrt{n-1}\big/i(n)^2\geq 1, \]
where the last inequality follows from $n\geq i(n)^4 + 1$.
For each $n\geq N$ let us denote \[D_n=E_n\cap\varphi^{-1}\Big[C_{j(n)}^{i(n)}\Big].\] We have $D_n\sub B_n$, $\mu_n(D_n)\geq 1$, and $\varphi[D_n]\sub C_{j(n)}^{i(n)}\sub A_{i(n)}$.

As $\varphi$ is finite-to-one, there is an increasing sequence $\seqkN{n_k}$ of natural numbers such that $i(n_k)< i(n_j)$ for every $k<j$. Let \[ X=\bigcup_{k\in\omega}D_{n_k}.\] 
Since for every $k\in\omega$ we have $\mu_{n_k}(X)\geq 1$, it follows that $X\notin\exh\big(\sup_{n\in\omega_+}\mu_n\big)$. Finally, we have \[\varphi[X]\sub\bigcup_{k\in\omega}C_{j(n_k)}^{i(n_k)},\] where the union is disjoint, and so $\lambda_{i(n_k)}(\varphi[X])\leq 1/i(n_k)$ for every $k\in\omega$ (and $\lambda_m(\varphi[X])=0$ for other $m\in\omega$). Thus, we get
$\lim_{n\to\infty}\lambda_n\big(\varphi[X]\big) =0$, due to the fact that $i(n_k)\to\infty$. Therefore, we have $\varphi[X]\in\exh\big(\sup_{n\in\omega_+}\lambda_n\big)$.

\end{proof}

\begin{corollary}
For every family $\{\II_{\alpha}\colon\alpha < \kappa\}\sub\AAA\NN$, where $\kappa<\unb$, there is $\JJ\in\AAA\NN$ such that $\II_{\alpha}<_K\JJ$ for all $\alpha<\kappa$.
\end{corollary}
\begin{proof}
If $\{\II_{\alpha}\colon\alpha < \kappa\}\sub\AAA\NN$ and $\kappa<\unb$, then, by Theorem \ref{thm_no_maximal}, for each $\alpha<\kappa$ there is $\JJ_{\alpha}\in\AAA\NN$ such that $\II_{\alpha}<_K\JJ_{\alpha}$. By Corollary \ref{bound}, there is $\JJ\in\AAA\NN$ satisfying $\JJ_{\alpha}\leq_K\JJ$ for every $\alpha<\kappa$, thus we have $\II_{\alpha}<_K\JJ$ for all $\alpha<\kappa$.
\end{proof}

\section{$\AAA\NN$ vs $\BB\JJ\NN\PP$}

We will now summarize differences between the classes of ideals $\AAA\NN$ and $\BB\JJ\NN\PP$. As we said in Introduction, these are connected with differences between the Nikodym property and the Grothendieck property of Boolean algebras. We start with the following result, stating that $\AAA\NN$ is a subclass of $\BB\JJ\NN\PP$. Note that it actually follows from Theorems \ref{thm_positive_measures} and \ref{thm_bjnp}, but for the convenience of the reader and the self-containment of the paper we provide here a direct (and more analytic) argument.

\begin{proposition} \label{nikodym_bjn}
If a filter $F$ does not have the Nikodym property, then it has the BJNP. Consequently, $\AAA\NN\sub\BB\JJ\NN\PP$.
\end{proposition}
\begin{proof}
Let $\seqnN{\mu_n}$ be an anti-Nikodym sequence of measures on $N_F$. As $\|\mu_n\|\rightarrow\infty$, by taking a suitable subsequence we may assume that $\|\mu_n\|\geq 1$ for each $n\in\omega$. We define a sequence $\seqnN{\lambda_n}$ of measures by setting $\lambda_n=\mu_n/\|\mu_n\|$ for each $n\in\omega$. We have $\|\lambda_n\|=1$ for all $n\in\omega$ and  $|\lambda_n(A)|=|\mu_n(A)|/\|\mu_n\|\leq|\mu_n(A)|/1\rightarrow 0$ for every $A\in\clopen(N_F)$. 

By \cite[Lemma 3.8.(2)]{MS}, the \cech--Stone compactification $\beta N_F$ is the Stone space of some Boolean algebra, thus in particular it is zero-dimensional.
For every $n\in\omega$ we define a measure $\overline{\lambda_n}$ on $\beta N_F$ by setting $\overline{\lambda_n}(A)=\lambda_n(A\cap N_F)$ for each $A\in Bor(\beta N_F)$. By the definition we have $\big\|\overline{\lambda_n}\big\|=1$ for each $n\in\omega$ and
$\overline{\lambda_n}(A)\rightarrow 0$ for every $A\in\clopen(\beta N_F)$. We will show that, moreover, for every $f\in C(\beta N_F)$ we have$\int_{\beta N_F}f d\overline{\lambda_n}\to 0$, i.e. $\seqnN{\overline{\lambda_n}}$ is a BJN-sequence of measures on $\beta N_F$.

Let us take arbitrary $f\in C(\beta N_F)$ and $\eps>0$. By the zero-dimensionality of $\beta N_F$ and the Stone--Weierstrass theorem, the linear span of the set $\big\{\chi_{A}\colon A\in\clopen(\beta N_F)\big\}$ is norm dense in $C(\beta N_F)$ with the supremum norm. Therefore, there exist $m\in\omega$, real numbers $\alpha_1,\ldots,\alpha_m$, and clopen subsets $A_1,\ldots,A_m$ of $\beta N_F$ such that $\|f-\sum_{i=1}^m \alpha_i\chi_{A_i}\|_{\infty}<\eps/2$. By the convergence of the sequence $\seqnN{\overline{\lambda_n}}$ on clopen sets, there is $N\in\omega$ such that for every $n\geq N$ we have:
\[\sum_{i=1}^m|\alpha_i|\cdot|\overline{\lambda_n}(A_i)|<\eps/2.\]
Then, for every $n\geq N$ it holds:
\[\Big|\int_{\beta N_F} f d\overline{\lambda_n}\Big|\leq \int_{\beta N_F}\Big|f-\sum_{i=1}^m \alpha_i\chi_{A_i}\Big| d\overline{\lambda_n} + \int_{\beta N_F}\Big|\sum_{i=1}^m \alpha_i\chi_{A_i}\Big| d\overline{\lambda_n} \leq\]
\[\leq\Big\|f-\sum_{i=1}^m \alpha_i\chi_{A_i}\Big\|_{\infty}\cdot\|\overline{\lambda_n}\| + \sum_{i=1}^m |\alpha_i|\cdot|\overline{\lambda_n}(A_i)| < \eps/2 + \eps/2 = \eps, \]
which proves that $\seqnN{\overline{\lambda_n}}$ is a BJN-sequence of measures on $\beta N_F$.

To show that $\seqnN{\lambda_n}$ is a BJN-sequence of measures on $N_F$, let $f\in C_b(N_F)$. By the universal property of the \cech--Stone compactification, there is a continuous extension $\overline{f}\in C(\beta N_F)$ of $f$. We proved that 
$\int_{\beta N_F}\overline{f} d\overline{\lambda_n}\to 0$, and by the definition of $\overline{\lambda_n}$'s we have $\int_{\beta N_F}\overline{f} d\overline{\lambda_n}=\int_{N_F}f d\lambda_n$ for every $n\in\omega$, thus $\int_{N_F}f d\lambda_n\to 0$. 
\end{proof}

Let us recall that our characterization of the Nikodym property via exhaustive ideals, theorem \ref{thm_AN_exh}, was inspired by the following characterization of the BJNP.

\begin{theorem}[Marciszewski and Sobota, \cite{MS}]\label{thm_bjnp}
Let $F$ be a free filter on $\omega$. Then, the following are equivalent:
\begin{enumerate}
\item $N_F$ has the BJNP,
\item there is a density submeasure $\varphi$ such that $F\sub\exh(\varphi)^*$,
\item there is a non-pathological lsc submeasure $\varphi$ such that $F\sub\exh(\varphi)^*$,
\item $F \leq_K \ZZ^*.$
\end{enumerate}
\end{theorem}

The only difference between conditions $(2)$ and $(3)$ of this theorem and the corresponding conditions of Theorem \ref{thm_AN_exh} is the requirement in Theorem \ref{thm_AN_exh} that the submeasure $\varphi$ is not bounded, i.e. $\varphi(\omega)=\infty$. However, this apparently small difference changes significantly the properties of families of ideals, for example the cofinal structure of $\BB\JJ\NN\PP$ is very simple, as $(\BB\JJ\NN\PP, \leq_K)$ has a maximal element by Theorem \ref{thm_bjnp}, namely $\ZZ$.

By Theorem \ref{thm_bjnp} and Corollary \ref{tot_bounded} we get:

\begin{corollary}\label{nikodym_and_bjn}
If $\varphi$ is a non-pathological lsc submeasure on $\omega$ and the ideal $\II=\exh(\varphi)$ is totally bounded, then $\II\in\BB\JJ\NN\PP\sm \AAA\NN$.
\end{corollary}

In particular, all the density ideals satisfying any of the conditions from Theorem \ref{nikodym_equiv} are elements of $\BB\JJ\NN\PP\sm \AAA\NN$. 

\begin{corollary}
$\BB\JJ\NN\PP\sm \AAA\NN$ contains $\con$ many Borel pairwise non-isomorphic ideals on $\omega$.
\end{corollary}
\begin{proof}
By \cite[Propositions 5 and 6]{KwelaEU}, there is a family of pairwise non-isomorphic \erdos--Ulam ideals of size $\con$. Every \erdos--Ulam as a density ideal is Borel. The result follows then from Corollary \ref{EU_Nikodym}.
\end{proof}

Another example of an element of $\BB\JJ\NN\PP\sm \AAA\NN$ is the ideal $\tr(\NN)$ (see e.g. \cite[page 580]{HHH}), which is of the form $\exh(\varphi)$ for some non-pathological submeasure $\varphi$, and is totally bounded by \cite[Proposition 3.18]{HHH}. The statement of Corollary \ref{equiv_density} does not hold for this ideal, as $\tr(\NN)<_K\ZZ$.
It follows from the following facts:
\begin{itemize}
 \item by \cite[Proposition 3.6]{HHH} we have $\tr(\NN)\leq_K\ZZ$;
 \item there cannot hold $\tr(\NN)\equiv_K\ZZ$, as it would imply that $\cov^*\big(\tr(\NN)\big)=\cov^*(\ZZ)$ by \cite[Proposition 3.1]{HHH} (see \cite[Section 1]{HHH} for the definition of $\cov^*$), but by \cite[Theorem 3.15]{HHH} this equality does not hold in the random model, and the \katetov\ order between Borel ideals is absolute (see e.g. \cite[Section 1.2]{Hru}).
\end{itemize}

\section{Examples}

In this section we present some examples of classes of ideals on $\omega$ which are connected to the class $\AAA\NN$.

\subsection{Summable ideals and large families of non-isomorphic filters}

Given $f\colon\omega\to[0, \infty)$ such that $\sum_{n\in\omega} f(n) = \infty$, the \textit{summable ideal} corresponding to $f$ is the ideal
 \[\II_f=\Big\{A\sub\omega\colon\sum_{n\in A} f(n) < \infty\Big\}.\]

\noindent The classical example is the ideal $\II_{1/n}$ corresponding to the function $f(n)=1/n$.  It is easy to see that 
$\II_f=\exh(\mu_f)$, where $\mu_f$ is a non-negative measure on $\omega$ defined for every $A\sub\omega$ by
\[\mu_f(A)=\sum_{n\in A} f(n).\]

Since by the definition we have $\mu_f(\omega)=\infty$, and $\mu_f$ is obviously a non-pathological lsc submeasure, by Theorem \ref{thm_AN_exh} we get the following corollary
\begin{corollary}\label{summable}
The dual filter $\II_f^*$ of any summable ideal $\II_f$ does not have the Nikodym property.
\end{corollary}

Let us denote by $\Sigma$ the family of all summable ideals. It appears that $(\Sigma,\leq_K)$---and so, by Corollary \ref{summable}, $(\AAA\NN,\leq_K)$---has a rich structure.
Recall that $\PP(\omega)/Fin$ is the quotient algebra of subsets of $\omega$ modulo finite sets, endowed with the ordering induced by $\sub^*$. It is well-known that $(\PP(\omega)/Fin, \sub^*)$ contains increasing  chains of size $\mathfrak{b}$ and antichains of size $\con$. The following result connects $(\PP(\omega)/Fin, \sub^*)$ and
$(\Sigma,\leq_K)$.

\begin{theorem}[\guzman\textrm{ and }\meza\,\cite{GuzMez}]
There is an order embedding of \\ $(\PP(\omega)/Fin, \sub^*)$ into $(\Sigma, \leq_K)$.
\end{theorem}

\begin{corollary}\label{many_summable}
There is an order embedding of $(\PP(\omega)/Fin, \sub^*)$ into $(\AAA\NN, \leq_K)$. In particular, $(\AAA\NN, \leq_K)$ contains increasing chains of size $\mathfrak{b}$ and antichains of size $\con$, consisting of summable ideals.
\end{corollary}

\noindent Recall that in Theorem \ref{dom} and Corollary \ref{decreasing} we have shown that $(\AAA\NN, \leq_K)$ contains dominating families of size $\dom$ and decreasing chains of size $\con^+$.

We now present a sketch of the construction of large families of non-isomorphic filters with and without the Nikodym property (the ideas follow closely \cite[Section 6.1]{MS}).

\begin{corollary}\label{many_many}
There exist families $\FF_1$ and $\FF_2$ of size $2^{\con}$, each consisting of pairwise non-isomorphic free filters on $\omega$ whose dual ideals are tall, such that:
\begin{enumerate}
 \item every $F\in\FF_1$ does not have the Nikodym property,
 \item every $F\in\FF_2$ has the Nikodym property.
\end{enumerate}
\end{corollary}

\begin{proof}
Let $\GG_2$ be the family of all ultrafilters on $\omega$ and let $F=\II_{1/n}^*$. We put:
\[ \GG_1=\{G\oplus F\colon G\in\GG_2\},\]
where $G\oplus F$ denotes the disjoint sum of $F$ and $G$.
By Corollary \ref{ultrafilter} all members of $\GG_2$ do have the Nikodym property. 
Since $F$ does not have the Nikodym property by Corollary \ref{summable}, and $G\oplus F\leq_K F$ for every $G\in\GG_2$ (see \cite[Section 1.2]{Hru}), by Proposition \ref{katetov_nikodym} all members of $\GG_1$ do not have the Nikodym property. Moreover, $\big(G\oplus F\big)^*=G^*\oplus F^*$ is tall for each $G\in\GG_2$, as $G^*$ and $F^*$ are tall.
As $|\GG_2|=|\GG_1|=2^{\con}$, and each family of pairwise isomorphic filters has cardinality $\leq\con$, for each $i=1,2$ we can select a subfamily $\FF_i\sub\GG_i$ consisting of $2^{\con}$ many pairwise non-isomorphic filters.
\end{proof}

\subsection{Simple density ideals}

The class of simple density ideals was introduced in \cite{Lodz} (cf. also \cite{Kwela}). For every function $g\colon\omega\to [0,\infty)$ satisfying $g(n)\to\infty$ and $n/g(n)\nrightarrow 0$ we define the \textit{simple density ideal} corresponding to $g$ as:
\[\ZZ_g=\Big\{A\sub\omega\colon\limsup_{n\to\infty}\frac{|A\cap n|}{g(n)}=0\Big\}.\]
It is proved in \cite[Theorem 3.2]{Lodz} that every $\ZZ_g$ is a density ideal. \cite[Proposition 10]{Kwela} presents a sufficient condition for a density ideal generated by a sequence of measures $\seqnN{\mu_n}$ to be a simple density ideal. Since the statement of this result is rather long and technical, we omit it here.
The ideal $\Phi(f)$ satisfies conditions $(i),(ii)$ and $(v)$ from \cite[Proposition 10]{Kwela} for any $f\in\Baire$. Conditions $(iv)$ and $(vi)$ are also satisfied for $\Phi(f)$ when $f$ is non-decreasing and $f(n)\to\infty$, as using the notation from that proposition we have $a_n=1/f(n)$. By the definition of $\Phi(f)$ we have also $\min\big(\supp(\mu_n^f)\big)=\sum_{i<n}f(i)$ for every $n\in\omega$, where $\seqnN{\mu_n^f}$ is the sequence of measures generating $\Phi(f)$. Thus, for
condition $(iii)$ to be satisfied, we need that
\[\frac{1}{f(n)}\cdot\sum_{i<n}f(i)\to 0.\]
Therefore, by \cite[Proposition 10]{Kwela} we have the following corollary.

\begin{corollary}
Let $f\in\Baire$ be non-decreasing and satisfy the condition $f(n)\big/\sum_{i<n}f(i)\to\infty$. Then, the ideal $\Phi(f)$ is equal to the simple density ideal $\ZZ_g$, where $g\colon\omega\to [0,\infty)$ is defined by $g(k)=f(n)$ for $k\in\omega$ when $k\in\supp(\mu_n)$.
\end{corollary}

It is proved in \cite[Theorem 3]{Kwela} that there is a family $\{\ZZ_{g_\alpha}\colon\alpha<\con\}$ such that $\ZZ_{g_\alpha}$ is not $\leq_K$-comparable to $\ZZ_{g_{\beta}}$ for any $\alpha\neq\beta$. By Theorem \ref{nikodym_equiv}, for every density ideal $\II$ such that $\II\notin\AAA\NN$ we have $\II\equiv_K\ZZ$, and so the $\leq_K$-antichain must lie inside the $\AAA\NN$ family:

\begin{corollary}\label{many_density}
$(\AAA\NN,\leq_K)$ contains antichains of size $\con$ consisting of density ideals.
\end{corollary}

\section{The Nikodym property of Boolean algebras}\label{section_algebras}
This section is devoted to an analysis of the connection between embedding of a space $N_F$ without the finitely supported Nikodym property into the Stone space of a given Boolean algebra $\AAA$ and the lack of the Nikodym property of $\AAA$.

We first need to introduce some notation (cf. \cite[Section 7]{MS}). Let $X$ be a compact Hausdorff space, $Y$ an infinite countable subset of $X$, and $x\in\overline{Y}\sm Y$. Let $Z_{X}(x, Y)$ be the space $Y\cup\{x\}$ endowed with the subspace topology.
By $\mathfrak{N}_{X}(x)$ we denote the collection of all $U\sub X$ such that $x\in\interior U$. We then put:
\[ \mathfrak{J}_{X}(x, Y)=\big\{U\cap Y\colon U\in\mathfrak{N}_{X}(x)\big\}.\]
Note that since  $x\in\overline{Y}\sm Y$, $\mathfrak{J}_{X}(x, Y)$ is a filter on the infinite countable set $Y$ containing all its co-finite subsets. If $f\colon\omega\to Y$ is a bijection, then $F=\big\{f^{-1}[V]\colon V\in\mathfrak{J}_{X}(x, Y)\big\}$ is a free filter on $\omega$---we will say in this case that $F$ is \textit{f-associated} to 
$\mathfrak{J}_{X}(x, Y)$. We also have a bijective continuous function $\varphi_{f}\colon N_F\to Z_{X}(x, Y)$ associated with $f$, which is defined by
$\varphi_f\restriction\omega=f$ and $\varphi_f\big(p_F\big)=x$.

\begin{proposition}\label{prop_associated}
Suppose $\AAA$ is a Boolean algebra. Let $Y$ be a countable subset of $St(\AAA)$ and $x\in\overline{Y}\sm Y$. Let $f\colon\omega\to Y$ be a bijection and $F$ the filter on $\omega$ $f$-associated to $\mathfrak{J}_{St(\AAA)}(x, Y)$.

Then, if $F$ does not have the Nikodym property, then the algebra $\AAA$ does not have the Nikodym property.
\end{proposition}
\begin{proof}
Let $\seqnN{\nu_n}$ be an anti-Nikodym sequence of measures on $N_F$. For each $n\in\omega$ we define a measure $\mu_n$ on $St(\AAA)$ by setting
\[\mu_n(B) = \nu_n\big(\varphi_f^{-1}[B]\big) \]
for every $B\in Bor\big(St(\AAA)\big)$. Then, we get that $\|\mu_n\|=\|\nu_n\|\to\infty$ and that for each $n\in\omega$ the support $\supp(\mu_n)$ is finite and contained in $Z_{St(\AAA)}(x, Y)$. 

Since $\varphi_{f}$ is continuous, for any $A\in\clopen\big(St(\AAA)\big)$ we have $\varphi_f^{-1}[A]\in\clopen(N_F)$, hence by the assumption that $\seqnN{\nu_n}$ is anti-Nikodym we get:
\[\mu_n(A)=\nu_n\big(\varphi_f^{-1}[A]\big)\to 0 \textrm{ as } n\to\infty.\]

For each $n\in\omega$ we define a measure $\lambda_n$ on the algebra $\AAA$ by the formula $\lambda_n(A) = \mu_n\big([A]_{\AAA}\big)$ for every $A\in\AAA$. By the properties of the sequence $\seqnN{\mu_n}$ we have
 $\|\lambda_n\|\to\infty$ and $\lambda_n(A)\to 0$ for every $A\in\AAA$. Therefore, $\seqnN{\lambda_n}$ is a pointwise bounded sequence of measures on $\AAA$ which is not uniformly bounded, and hence $\AAA$ does not have the Nikodym property.
\end{proof}

\begin{theorem}
Suppose $\AAA$ is a Boolean algebra and $G$ is a filter on $\omega$. Let $\varphi\colon N_G\to St(\AAA)$ be a continuous function such that $\varphi^{-1}\big(\varphi(p_G)\big)=\{p_G\}$ (e.g. $\varphi$ is an injection). 

Then, if $G$ does not have the Nikodym property, then $\AAA$ does not have the Nikodym property.
\end{theorem}
\begin{proof}
Let us define $Y=\varphi[\omega]$ and $x=\varphi(p_G)$. It follows that $Y$ is an infinite countable subset of $St(\AAA)$, as if $Y$ was finite, in particular closed in $St(\AAA)$, then it would imply that $\varphi^{-1}[Y]=\omega$ is closed in $N_F$. 
We also have $x\in\overline{Y}\sm Y$. Let $f\colon\omega\to Y$ be a bijection and let $F$ be the filter on $\omega$ $f$-associated to $\mathfrak{J}_{X}(x, Y)$. We define $g\colon\omega\to\omega$ by the formula 
$g(n)=f^{-1}\big(\varphi(n))$ for every $n\in\omega$, i.e. $g=f^{-1}\circ(\varphi\restriction\omega)$. 

We claim that the function $g$ is a witness for $F\leq_K G$. Let $A\in F$. Since $F$ is $f$-associated to $\mathfrak{J}_{St(\AAA)}(x, Y)$, we have $f[A]\in\mathfrak{J}_{St(\AAA)}(x, Y)$, and thus $f[A]\cup\{x\}$ contains an open neighbourhood of $x$ in $Z_{St(\AAA)}(x, Y)$. Since $\varphi$ is continuous and $\varphi^{-1}(\{x\})=\{p_G\}$, the set
$\varphi^{-1}\big[f[A]\cup\{x\}\big]=\varphi^{-1}\big[f[A]\big]\cup\{p_G\}$ contains an open neighbourhood of $p_G$ in $N_G$, and so $g^{-1}[A]=\varphi^{-1}\big[f[A]\big]\in G$.

Therefore, $F\leq_K G$. By Proposition, \ref{katetov_nikodym} the filter $F$ does not have the Nikodym property, thus by Proposition \ref{prop_associated} the algebra $\AAA$ does not have the Nikodym property.
\end{proof}

\begin{corollary}\label{corollarynikodymalgebras}
Let $\AAA$ be a Boolean algebra and let $G$ be a filter on $\omega$ that does not have the Nikodym property. If $N_G$ homeomorphically embeds into $St(\AAA)$, then $\AAA$ does not have the Nikodym property.
\end{corollary}

The following corollary follows from Theorem \ref{nikodym_equiv} and Corollary \ref{summable}.

\begin{corollary}
 Let $\AAA$ be a Boolean algebra and let $\II$ be a summable ideal on $\omega$, or a density ideal which is not totally bounded. If $N_{\II^*}$ homeomorphically embeds into $St(\AAA)$, then $\AAA$ does not have the Nikodym property.
\end{corollary}

\begin{remark}
Let us note that embedding a suitable space $N_F$ into the Stone space is not the only reason for a Boolean algebra to lack the Nikodym property. There exists an algebra $\SSS$ which does not have the Nikodym property and such that $St(\SSS)$ does not contain any homeomorphic copy of a space $N_F$ for which $F\in\BB\JJ\NN\PP$ (see \cite[Example 4.10]{Sch} and \cite[Example 7.8]{MS}).
\end{remark}

The family of filters $\FF_1$ constructed in Corollary \ref{many_many} gives us a rich family of pairwise non-homeomorphic spaces $N_F$ which forbid the Nikodym property of Boolean algebras.
By \cite[Propositions 3.10.(2) and 4.3]{MS}, for each $F\in\FF_1$ the space $N_F$ contains no non-trivial convergent sequence.
Therefore, by Corollary \ref{corollarynikodymalgebras}, we get the following result.

\corollarymanyAN

\section{Final remarks and questions}

\subsection{Non-totally bounded non-pathological ideals}
By Corollary \ref{tot_bounded}, if $\II=\exh(\varphi)$ is a totally bounded ideal for some lsc submeasure $\varphi$, then $\II^*$ has the Nikodym property. By Remark \ref{nik_tot_bounded}, the converse does not hold, but the counterexamples are constructed only with pathological submeasures. Theorem \ref{nikodym_equiv} shows that the opposite implication is true for all density submeasures.
We do not know if this implication holds for all non-pathological submeasures.

\begin{question}
Does there exist a non-pathological lsc submeasure $\varphi$ such that the ideal $\exh(\varphi)$ is not totally bounded and $\exh(\varphi)^*$ has the Nikodym property?
\end{question}

\subsection{Tukey reductions}
The analysis of the cofinal structure of the order $(\AAA\NN, \leq_K)$ shows that it is, in some way, similarly complicated as the cofinal structure of $(\Baire, \leq^*)$. The standard way of comparing cofinal structures of two orders is through Tukey reductions, see e.g. \cite{Frem} and \cite{SolTod}.

Let $(P, \leq_P)$ and $(Q, \leq_Q)$ be partial orders. A function $f\colon P\to Q$ is said to be a \textit{Tukey reduction} if for every $q_0\in Q$ there exists $p_0\in P$ such that, for every $p\in P$, the following implication holds:
\[f(p)\leq_Q q_0 \: \Longrightarrow \: p \leq_P p_0, \]
i.e. if the preimages by $f$ of bounded subsets of $Q$ are bounded subsets of $P$. We write $P\preceq_T Q$ if there exists a Tukey reduction from $P$ to $Q$. We say that $P$ and $Q$ are \textit{Tukey equivalent}, which we denote by $P\equiv_T Q$, if $P\preceq_T Q$ and $Q\preceq_T P$.

An example of using Tukey reductions for classes of ideals ordered by $\leq_K$ is the following result about the class of all $F_\sigma$ ideals on $\omega$.

\begin{theorem}[Minami and Sakai \cite{MinSak}]
\[ (F_\sigma\textit{-ideals},\leq_K) \equiv_T (F_\sigma\textit{-ideals},\leq_{KB}) \equiv_T (\Baire, \leq^*). \]
\end{theorem}

Using Theorem \ref{thm_katetov} it is not difficult to prove the following reduction.

\begin{proposition}\label{tukey}
$(\AAA\NN,\leq_K)\preceq_T (\Baire, \leq^*).$
\end{proposition}
\begin{proof}
We define a Tukey reduction $\Psi\colon\AAA\NN\to\Baire$ in the following way. Let $\II\in\AAA\NN$. By Theorem $\ref{thm_katetov}.(1)$ there exists $f\in\Baire$ such that
$\II\leq_K\Phi(f)$. Let us fix any such $f\in\Baire$, and put $\Psi(\II)=g$, where $g(n)=2n^2\cdot f(n)$ for every $n\in\omega$.

To prove that $\Psi$ is a Tukey reduction, let us take any $h\in\Baire$. We need to show that the following implication holds for every $\II\in\AAA\NN$:
\[\Psi(\II)\leq^* h \: \Longrightarrow \: \II\leq_K \Phi(h). \]

By the definition of $\Psi$, if $\Psi(\II)\leq^* h$ then there exists $f\in\Baire$ such that $\II\leq_K\Phi(f)$ and $2n^2\cdot f(n)\leq h(n)$ holds for all but finitely many $n\in\omega$.
By Theorem $\ref{thm_katetov}.(2)$ we get $\Phi(f)\leq_K \Phi(h)$, and so $\II\leq_K\Phi(h)$.
\end{proof}
 
Note that Corollaries \ref{dom} and \ref{bound} both follow directly from Proposition \ref{tukey} by \cite[Theorem 1J]{Frem}. We do not know if reverse Tukey reduction exist.

\begin{question}
Does it hold $(\Baire, \leq^*)\preceq_T (\AAA\NN,\leq_K)$?
\end{question}

\end{document}